\pdfoutput=1

\documentclass[12pt,
headings=small,
]{scrartcl}

\usepackage[a4paper, inner=4cm, outer=4cm]{geometry} %
\setkomafont{sectioning}{\rmfamily\bfseries} %

\usepackage{pifont}
\usepackage{times}
\usepackage{amsfonts}       %
\usepackage{amsmath} %
\usepackage{amssymb}        %
\usepackage{amstext}        %
\usepackage{amsthm}         %
\usepackage{booktabs}
\usepackage[obeyDraft]{todonotes}
\usepackage{latexsym}
 \usepackage[matrix,curve]{xypic}
\usepackage{microtype}        %
\usepackage{titlefoot}
\usepackage{pdflscape}

\usepackage{placeins}

\makeatletter
\renewcommand*\env@matrix[1][*\c@MaxMatrixCols c]{%
  \hskip -\arraycolsep
  \let\@ifnextchar\new@ifnextchar
  \array{#1}}
\makeatother

\newcommand{\R}{\mathbb{R}}
\newcommand{\C}{\mathbb{C}}

\newcommand{\CC}{\mathbb{C}}

\newcommand{\CCC}{\mathcal{C}}

\DeclareMathOperator{\eval}{eval}

\newtheorem{thm}{Theorem}[section]
\newtheorem{cor}[thm]{Corollary}
\newtheorem{lem}[thm]{Lemma}
\newtheorem{prop}[thm]{Proposition}

\theoremstyle{definition}
\newtheorem{defn}[thm]{Definition}

\newtheorem{rmk}[thm]{Remark}

\newcommand{\diag}{\mathrm{diag}}

\newcommand{\St}{\mathrm{St}}
\newcommand{\fin}{\mathrm{fin}}

\newcommand{\crit}{\mathrm{crit}}

\newcommand{\Sl}{\mathrm{Sl}}
\newcommand{\Gl}{\mathrm{Gl}}
\newcommand{\GSp}{\mathrm{GSp}}
\newcommand{\id}{\mathrm{id}}

\newcommand{\sreg}{\mathrm{sreg}}
\newcommand{\reg}{\mathrm{reg}}
\newcommand{\PS}{\mathrm{PS}}
\newcommand{\ex}{\mathrm{ex}}
\newcommand{\Diff}{\;\mathrm{d}}
\renewcommand{\Re}{\mathrm{Re}}

\newcommand\nosf[1]{\begin{footnotesize}\textup{\textsf{#1}}\end{footnotesize}}

\DeclareMathOperator{\Hom}{Hom}
\DeclareMathOperator{\Ext}{Ext}

\DeclareMathOperator{\coker}{coker}
\DeclareMathOperator{\ind}{ind}
\DeclareMathOperator{\Ind}{Ind}

\usepackage{enumitem}
\setlist[enumerate]{itemsep=-0.5ex plus0.1ex minus 0.2ex}
\setlist[description]{itemsep=-0.5ex plus0.1ex minus 0.2ex}
\setlist[itemize]{itemsep=-0.5ex plus0.1ex minus 0.2ex}
\let\svthefootnote\thefootnote
\newcommand\blankfootnote[1]{%
  \let\thefootnote\relax\footnotetext{#1}%
  \let\thefootnote\svthefootnote%
}

\everydisplay{
\thickmuskip=5mu plus 2mu
\medmuskip=4mu plus 1mu
\relax}

\title{\begin{large}
Spinor Euler factors for $\GSp(4)$ in the subregular case
\end{large}}
\author{Mirko~R\"osner \and Rainer Weissauer}
\date{}
\begin{document}
\maketitle

\begin{abstract}
For local non-archimedean fields $k$, Piatetski-Shapiro has defined local spinor $L$-factors for irreducible representations $\Pi$ of $\GSp(4,k)$ of dimension $>1$, attached to a choice of a Bessel model $\Lambda$.
We classify regular poles that do not come from the asymptotic of the Bessel functions in the Bessel model.
For anisotropic Bessel models there are no such subregular poles.
\blankfootnote{2010 \textit{Mathematics Subject Classification.} Primary 22E50; Secondary 11F46, 11F70, 20G05.} 
\blankfootnote{\textit{Key words and phrases.} Symplectic group, Bessel model, spinor $L$-function, regular poles.}
\end{abstract}

\section{Introduction}
Let $k$ be a local non-archimedean field and $\mu$ a smooth character of $k^\times$.
For irreducible smooth representations $\Pi$ of dimension $>1$ of the symplectic group of similitudes $G=\mathrm{GSp}(4,k)$,
Piatetski-Shapiro~\cite{PS-L-Factor_GSp4} has constructed local $L$-factors
$$L^{\mathrm{PS}}(s,\Pi,\mu,\Lambda)$$
attached to a choice of a Bessel model $\Lambda$ of $\Pi$.
It was already expected by Piatetski-Shapiro and Soudry \cite{Soudry_Piatetski_L_Factors} that these $L$-factors do not depend on the choice of the Bessel model.
For certain Borel induced irreducible representations $\Pi$, they proved this \cite[Thm.\,2.4]{Soudry_Piatetski_L_Factors}.
The local $L$-factor $L^{\mathrm{PS}}(s,\Pi,\mu,\Lambda)$ factorizes into a regular and an exceptional part.
Piatetski-Shapiro stated that generically the poles of the regular part should come from the asymptotic of the Bessel functions.
However, the precise conditions of this expectation were left open in the formulation of \cite[thm.~4.1]{PS-L-Factor_GSp4}.
Notice, more or less by definition, the poles coming from the asymptotic of the Bessel functions form the $L$-factor of the Bessel module \cite{RW}.
Dani\c{s}man has shown that Piateskii-Shapiro's expectation holds true for the case of anisotropic Bessel models \cite[proposition~2.5]{Danisman}.
Dani\c{s}man \cite{Danisman, Danisman_Annals, Danisman2, Danisman3} and the authors \cite{Anisotropic_Exceptional} explicitly determined the regular and exceptional $L$-factors.

\begin{table}
\begin{footnotesize}
\caption{The subregular $L$-factor for split Bessel models $\Lambda=\rho\boxtimes\rho^\divideontimes$.\label{tab:subregular_poles}}
\begin{center}
\begin{tabular}{llll}
\toprule
Type       & $\Pi$ & Condition & $L_{\sreg}^{\PS}(s,\Pi,1,\Lambda)$ \\
\midrule
\nosf{I}   & $\chi_1\times\chi_2\rtimes\sigma$  & $\rho\in\Delta_-(\Pi)$ & $L(s,\nu^{1/2}\rho)$ \\
           &         & $\rho^\divideontimes\in\Delta_-(\Pi)$ & $L(s,\nu^{1/2}\rho^\divideontimes)$ \\
\nosf{IIa} & $\St(\chi_1)\rtimes \sigma$ & $\rho\in\Delta_-(\Pi)$ & $L(s,\nu^{1/2}\rho)$ \\
           &         & $\rho^\divideontimes\in\Delta_-(\Pi)$ & $L(s,\nu^{1/2}\rho^\divideontimes)$ \\
\nosf{Va}  & $\delta([\xi,\nu\xi],\nu^{-1/2}\sigma)$& $\rho\in\Delta_-(\Pi)$ & $L(s,\nu^{1/2}\rho)$\\
           &                      & $\rho^\divideontimes\in\Delta_-(\Pi)$ & $L(s,\nu^{1/2}\rho^\divideontimes)$\\
\nosf{VIa} & $\tau(S,\nu^{-1/2}\sigma)$& $\rho\in\Delta_-(\Pi)$ & $L(s,\nu^{1/2}\rho)$\\
           &                      & $\rho^\divideontimes\in\Delta_-(\Pi)$ & $L(s,\nu^{1/2}\rho^\divideontimes)$\\
\nosf{X}   & $\pi\rtimes\sigma$  & $\rho\in\Delta_-(\Pi)$ & $L(s,\nu^{1/2}\rho)$\\
           &                      & $\rho^\divideontimes\in\Delta_-(\Pi)$ & $L(s,\nu^{1/2}\rho^\divideontimes)$\\
\nosf{XIa} &$\delta(\nu^{1/2}\pi,\nu^{-1/2}\sigma)$ & $\rho\in\Delta_-(\Pi)$ & $L(s,\nu^{1/2}\rho)$ \\
           &                      & $\rho^\divideontimes\in\Delta_-(\Pi)$ & $L(s,\nu^{1/2}\rho^\divideontimes)$\\
\nosf{IIIb}& $\chi_1\rtimes \sigma\mathbf{1}_{\GSp(2)}$ & $\rho\in\{\sigma,\chi_1\sigma\}$     & $L(s,\nu^{1/2}\chi_1\sigma)L(s,\nu^{1/2}\sigma)$\\
\nosf{IVc} & $L(\nu^{3/2}\St,\nu^{-3/2}\sigma)$ & $\rho\in\{\nu^{\pm1}\sigma\}$ & $L(s,\nu^{3/2}\sigma)$\\
\nosf{VIc} & $L(\nu^{1/2}\St,\nu^{-1/2}\sigma)$ & $\rho=\sigma$                 & $L(s,\nu^{1/2}\sigma)$\\
\nosf{VId} & $L(\nu,1\rtimes\nu^{-1/2}\sigma)$  & $\rho=\sigma$                 & $L(s,\nu^{1/2}\sigma)$\\
           & otherwise                          &                               &$1$\\
\bottomrule 
\end{tabular}
\end{center}
\end{footnotesize}
\end{table}

In this paper we consider the case of split Bessel models.
In contrast to the anisotropic case certain regular poles are not determined by the asymptotics of the Bessel functions \cite[prop.~6.16]{RW} and we call them \emph{subregular poles}. See section~\ref{s:filtration_of_poles} for the precise definition and for the remaining regular and exceptional poles, see \cite{RW} and \cite{W_Excep}.
We determine the subregular poles for every irreducible $\Pi$ and every split Bessel model, see table~\ref{tab:subregular_poles}. It turns out that there is a one-to-one correspondence between subregular poles and functionals $$\ell:\Pi\to(\CC,\widetilde{\rho})$$ that are equivariant with respect to a certain subgroup $H_+$ of $G$ and a character $\widetilde{\rho}$ of $H_+$ related to the Bessel model. We completely classify these functionals in section~\ref{s:H_+-functionals} and show that they satisfy a multiplicity one property.
Our results clarify the above-mentioned expectations by Piatetski-Shapiro.
Surprisingly, the product of Tate factors that comes from the asymptotic does depend on the choice of the Bessel model.
Only together with the Tate factors from the subregular poles, they define the regular part $L_\reg^\PS$, which is independent of the split Bessel model.
Since the exceptional part is also independent of the split Bessel model, this implies the independence of the full $L$-factor \cite{W_Excep, Anisotropic_Exceptional}.
In the following, we freely use notation and concepts from our previous work on the subject.

It is our main result that the $L$-factor of Piatetski-Shapiro is independent of the choice of a split Bessel model and coincides with the $L$-factor of the representation of the Weil-Deligne group attached to $\Pi$ via the local Langlands correspondence, see table~A.8 of Roberts and Schmidt \cite{Roberts-Schmidt}.
In particular our results are in complete accordance with the local Langlands correspondence for $\GSp(4)$ as established by Gan and Takeda \cite{Gan_Takeda_LLC}. Note that these authors use a different construction for the $L$-factors.
Piatetski-Shapiro's $L$-factors play an important role in the construction of Euler systems for $\GSp(4)$ by Loeffler, Skinner and Zerbes \cite{LSZ_EulerGSp4}, inspired by previous results of Lemma \cite{Lemma08, Lemma17}.
The authors are very grateful to David Loeffler for his comments and suggestions.

\tableofcontents

\subsection{Piatetski-Shapiro's local spinor $L$-factor for $\GSp(4)$}\label{s:filtration_of_poles}
The group of symplectic similitudes $\GSp(4)$ is defined by the equation
$$gJg^t=\lambda_G(g) J$$ for $g\in\Gl(4)$ and $\lambda_G(g)\in \Gl(1)$ where we use Siegel's notation $J=\left(\begin{smallmatrix}0&I_2\\-I_2&0\end{smallmatrix}\right)\,$.
Fix an irreducible smooth representation $\Pi$ of $G=\GSp(4,k)$ with central character $\omega$.
Suppose there is a smooth character $\Lambda=\rho\boxtimes\rho^\divideontimes$ of the split torus
$$\widetilde{T}\quad= \quad\{\diag(t_1,t_2,t_2,t_1)\in G\}\quad\cong\quad \Gl(1)\times \Gl(1)\ $$
that defines a non-zero $(\widetilde{T}N,\Lambda\boxtimes\psi)$-equivariant functional of $\Pi$.
Frobenius reciprocity yields an embedding $\Pi\hookrightarrow \Ind_R^G(\Lambda\boxtimes\psi)$, $v\mapsto W_v$ and its image is a split Bessel model for $\Pi$.
Attached to $v$ and a Schwartz-Bruhat function $\Phi\in C_c^\infty(k^4)$ is the zeta-function
\begin{equation*}
 Z^{\PS}(s,v,\Lambda,\Phi,\mu)=\int_{\widetilde{N}\backslash H}W_v(h)\Phi((0,0,1,1)h)\mu(\lambda_G(h))
 |\lambda_G(h)|^{s+\tfrac12}\Diff h\ .
\end{equation*}
The integral converges for sufficiently large $\Re(s)>0$ and admits a unique meromorphic continuation to the complex plane. The local spinor $L$-factor $$L^{\PS}(s,\Pi,\mu,\Lambda)$$ attached to $\Pi$ and $\Lambda$
by Piatetski-Shapiro~\cite{Soudry_Piatetski_L_Factors} is the regularization $L$-factor of $Z^{\PS}(s,v,\Lambda,\Phi,\mu)$ varying over $v\in\Pi$ and $\Phi\in C_c^\infty(k^4)$.
The isomorphism $k^2\times k^2\to k^4\,,\ ((x_1,y_1),(x_2,y_2))\mapsto (x_1,x_2,y_1,y_2)$ defines inclusions
$$0\ \hookrightarrow\  C_c^\infty(k^2\setminus \{0\} \times k^2\setminus \{0\}) \ \hookrightarrow\   C_c^\infty(k^4\setminus \{0\})\ \hookrightarrow\   C_c^\infty(k^4)$$
and thus a filtration of $ C_c^\infty(k^4)$ with quotients isomorphic to
$$ C_c^\infty(k^2\setminus \{0\})\otimes  C_c^\infty(k^2\setminus \{0\})\quad ,\quad  C_c^\infty(k^2\setminus \{0\})\oplus C_c^\infty(k^2\setminus\{0\})\quad , \quad \C\ .$$
The regularization $L$-factor of the family of zeta-functions $Z^{\PS}(s,v,\Lambda,\Phi,\mu)$ varying over $v\in\Pi$ and $\Phi\in C_c^\infty(k^2\setminus \{0\}\times k^2\setminus\{0\})$
coincides with the regularization $L$-factor of the \emph{regular zeta-functions}
\begin{equation*}Z_{\reg}^{\PS}(s,W_v,\mu)=\int_{T} W_v(x_\lambda) \mu(\lambda)|\lambda|^{s-\tfrac32} \Diff x_\lambda\quad ,\quad v\in\Pi\,,\end{equation*}
again convergent for sufficiently large $\Re(s)>0$ with unique meromorphic continuation, see \cite[prop.~6.16]{RW} and compare Dani\c{s}man~\cite{Danisman_Annals} in the anisotropic case.
This $L$-factor equals $L(s,\mu\otimes M)$ for
$$M=\nu^{-3/2}\otimes\beta_\rho(\Pi)/\beta_\rho(\Pi)^S\ .$$
It has been determined explicitly by the authors \cite{RW} for every $\Pi$ and $\Lambda$.

The \emph{subregular factor} $L_{\sreg}^{\PS}(s,\Pi,\mu,\Lambda)$ is the regularization $L$-factor of
\begin{equation*}
 \frac{Z^{\PS}(s,v,\Lambda,\Phi,\mu)}{L(\mu\otimes M, s)}\ ,
\end{equation*}
varying over $v$ and $\Phi\in C_c^\infty(k^4\setminus \{0\})$. In this work we determine the subregular factor explicitly as a product of Tate factors.

The \emph{exceptional factor} $L_{\ex}^{\PS}(s,\Pi,\mu,\Lambda)$ is the regularization $L$-factor of 
\begin{equation*}
\frac{Z^{\PS}(s,v,\Lambda,\Phi,\mu)}{L(\mu\otimes M, s) L^{\PS}_{\sreg}(s,\Pi,\mu,\Lambda)}\ ,
\end{equation*}
varying over $v$ and $\Phi\in C_c^\infty(k^4)$. For the explicit values, see \cite{W_Excep}. %
Hence the $L$-factor of Piatetski-Shapiro factorizes
\begin{equation*}
 L^{\PS}(s,\Pi,\mu,\Lambda) =  L^{\PS}_{\reg}(s,\Pi,\mu,\Lambda) L^{\PS}_{\ex}(s,\Pi,\mu,\Lambda) \ 
\end{equation*}
where the \emph{regular factor} is
$$ L^{\PS}_{\reg}(s,\Pi,\mu,\Lambda) = L(\mu\otimes M, s) L^{\PS}_{\sreg}(s,\Pi,\mu,\Lambda) \ .$$
It turns out that both the regular and the exceptional factor are independent of the split Bessel model.

\subsection{Notation}\label{s:notation}

Fix a local non-archimedean number field $k$ of characteristic not two with ring of integers $\mathfrak{o}_k$ and its split quadratic extension $K=k\times k$.
We normalize the additive Haar measure on $k$ so that $\mathrm{vol}(\mathfrak{o})=1$ and the multiplicative Haar measure on $k^\times$ so that $\mathrm{vol}(\mathfrak{o}^\times)=1$.

The group $G=\GSp(4,k)$ of symplectic similitudes $g\in \Gl(4,k)$
is defined
by the equation
$gJg^t=\lambda_G(g) J$ with a scalar $\lambda_G(g)\in \Gl(1)$ where we use Siegel's notation $ J=\left(\begin{smallmatrix}0&I_2\\-I_2&0\end{smallmatrix}\right)\ .$
The center $Z_G=\{\lambda\cdot I_4 \mid \lambda\in Gl(1)\}$ is the subgroup of scalar matrices.
Fix the standard Siegel parabolic $P$ and the standard Klingen parabolic $Q$
$$P= \left\{\begin{pmatrix}
\ast & \ast & \ast &\ast \\
\ast & \ast & \ast &\ast \\
0    & 0    & \ast &\ast \\
0    & 0    & \ast &\ast \\
\end{pmatrix}\right\}\cap G\ \ ,\qquad
Q=\left\{\begin{pmatrix}
\ast & \ast & \ast &\ast \\
0    & \ast & \ast &\ast \\
0    & 0    & \ast &0    \\
0    & \ast & \ast &\ast \\
\end{pmatrix}\right\}\cap G$$ and the standard Borel $B_G=P\cap Q$.
The subgroups $T$, $\widetilde{T}$, $\widetilde{N}=S_A\times S_C$, $S$ of $G$ are defined in \cite{RW} and $N=\widetilde{N}S$ is the unipotent radical of $P$.
We denote by $Q_1$ the subgroup of $Q$ of matrices with third diagonal entry $1$.
The Weyl group of $G$ has order eight and is generated by
\begin{equation*}
\mathbf{s}_1=\begin{pmatrix}0&1&0&0\\1&0&0&0\\0&0&0&1\\0&0&1&0\end{pmatrix}\qquad,\qquad
\mathbf{s}_2=\begin{pmatrix}1&0&0&0\\0&0&0&1\\0&0&1&0\\0&-1&0&0\end{pmatrix}\ .
\end{equation*}
The fiber product over the determinant $\Gl(2)\times_{\Gl(1)} \Gl(2)$
is identified with its image $H$ in $G$ under the embedding
\begin{equation*}
\left(\begin{pmatrix}a_1&b_1\\c_1&d_1\end{pmatrix},\begin{pmatrix}a_2&b_2\\c_2&d_2\end{pmatrix}\right)\longmapsto
\begin{pmatrix}
a_1& 0  &b_1&0  \\
0  & a_2&0  &b_2\\
c_1& 0  &d_1&0  \\
0  & c_2&0  &d_2
\end{pmatrix} \in G\ .
\end{equation*}

In the following $\Gl(2)$ will always denote the image of $\Gl(2)$ under the embedding $\Gl(2)\hookrightarrow H$ via $g\mapsto (\diag(\det(g),1),g)$.
The maximal parabolic subgroup $H_+\subseteq H$ defined by $c_1=0$ factorizes as $$H_+ = Z_G\cdot \Gl(2)\cdot S_A\ .$$

The affine linear group is defined as the mirabolic subgroup
\begin{equation*}
\Gl_a(n)=\left\{\begin{bmatrix}[c|c]g & v\end{bmatrix}:=\begin{pmatrix}g&v\\0&1\end{pmatrix} \mid g\in \Gl(n) ,\ v\in k^{n}\right\}
\subseteq \Gl(n+1)\ .
\end{equation*}

In the following we only consider the groups of $k$-valued points. By an abuse of notation we write $\Gl(n)$ or $\GSp(4)$ meaning $\Gl(n,k)$ resp. $\GSp(4,k)$.

\section{Subregular poles}
A subregular pole is a pole of the subregular $L$-factor $L_{\sreg}^{\PS}(s,\Pi,\mu,\Lambda)$.

\begin{prop}\label{2.1}
The subregular poles of $\Pi$ are the poles of the functions
$$\frac{Z_{\reg}^{\PS}(s,W_v,\mu)}{L(\mu\otimes M, s )} L(s,\nu^{1/2}\rho\mu)
\qquad\text{and}\qquad
\frac{Z_{\reg}^{\PS}(s,W_v,\mu)}{L(\mu\otimes M, s )} L(s,\nu^{1/2}\rho^\divideontimes\mu) $$
for $v\in\Pi$ invariant under the compact group $\id\times\Sl(2,\mathfrak{o}_k) \subseteq H$.
\end{prop}

\begin{proof}
We set $\mu=1$ by a suitable twist.
Recall that a subregular pole is a pole of
$Z^{\PS}(s,v,\Lambda,\Phi,\mu)/L(M, s )$
for some $\Phi$ with $\Phi(0,0,0,0)=0$.
Without loss of generality, we can assume
\begin{enumerate}
\item $\Phi(x_1,x_2,y_1,y_2)=\Phi_1(x_1,y_1)\Phi_2(x_2,y_2)$ for $x,y\in K$, because these $\Phi$ span the tensor product $ C_c^\infty(K\oplus K)= C_c^\infty(k\oplus k)\otimes  C_c^\infty(k\oplus k)$,
\item $\Phi_2(0,0)\neq0$ but $\Phi_1(0,0)=0$, otherwise apply the Weyl involution $\mathbf{s}_1$ 
that switches $\rho$ and $\rho^\divideontimes$, compare \cite[lemma~4.17]{RW},
\item $\Phi_2$ is the characteristic function of $\mathfrak{o}_k\oplus \mathfrak{o}_k$, otherwise add a suitable Schwartz function with support in $k^2\setminus \{0\}\times k^2\setminus \{0\}$.
\end{enumerate}

By Iwasawa decomposition, $H=H_+ \cdot (\Sl(2,\mathfrak{o}_k)\times\{\id\})$, so
for a suitable choice of Haar measure and sufficiently large $\Re(s)$,
\begin{equation*}
 Z^{\PS}(s,v,\Lambda,\Phi,1)= \int_{\Sl(2,\mathfrak{o}_k)} Z^\PS_+(s,\Pi(k_1,1)v,\Lambda,\Phi_1^{k_1}\otimes\Phi_2) \Diff k_1
\end{equation*} with $\Phi_1^{k_1}(x_1,y_1)=\Phi_1((x_1,y_1)k_1)$
and the partial zeta integral
\begin{equation*}
Z^{\PS}_+(s,v,\Lambda,\Phi,1) = \int_{\widetilde{N}\backslash H_+}W_v(h) \Phi((0,0,1,1)h)
 |\lambda_G(h)|^{s+\tfrac12} \delta_{H_+}^{-1}(h)\Diff_R h \ .
\end{equation*}
with a right Haar measure $\Diff_R h$ on $H_+$, see \cite[2.1.5(ii)]{Bump} and \cite[1.21]{Bernstein-Zelevinsky76}.
The modulus factor is defined by $\delta_{H_+}(g) = |\det(g)|$ for $g\in \Gl(2)$ embedded into $Q_1$ as above. $\delta_{H_+}$ is trivial on the center $Z_G$ and on $\widetilde{N}$.
It is defined such that $\delta_{H_+}^{-1}(h)\Diff_R h$ a left-invariant Haar measure.

The $k_1$-integral is a finite sum, so every pole of $Z^\PS(s,v,\Lambda,\Phi,\mu)$ comes from a pole of the partial zeta integral.
Conversely, $\Phi_1$ can be chosen arbitrarily in $ C_c^\infty(k^2\setminus \{0\})$, so every pole of the partial zeta integral gives rise to a pole of $Z^{\PS}(s,v,\Lambda,\Phi,\mu)$ for some $v$ and $\Phi$.
Hence the subregular poles are the poles of the partial zeta integrals varying over $\Phi_1$ and $v$.

By Iwasawa decomposition, every $h\in H_+$ factorizes as
$h=\widetilde{n}x_\lambda \widetilde{t}(1,k_2)\in \widetilde{N} T \widetilde{T} (1\times\Sl(2,\mathfrak{o}_k)$
for $\widetilde{t}=\diag(t_1,t_2,t_2,t_1)$. The Bessel function transforms with $W_v(\widetilde{t}t(1,k_2))=\rho(t_1)\rho^\divideontimes(t_2)W_{\Pi(1,k_2)v}(t)$,
so the partial zeta integral is
\begin{gather*}
 \int_{\Sl(2,\mathfrak{o}_k)} Z^{\PS}_{\reg}(s,W_{\Pi(1,k_2)v})\Diff k_2 \int_{\widetilde{T}}\Phi_1((0,t_2))\Phi_2((0,t_1)) \Lambda(\widetilde{t})|t_1t_2|^{s+1/2}\Diff \widetilde{t} = \\
 Z^{\PS}_{\reg}(s,W_v^{\mathrm{av}}) \int_{k^\times}\Phi_1(0,t_2)\rho^\divideontimes(t_2) |t_2|^{s+1/2}\Diff t_2
\int_{k^\times}\Phi_2(0,t_1)\rho(t_1) |t_1|^{s+1/2}\Diff t_1\ 
\end{gather*}
with the average $W_v^{\mathrm{av}}(g)=\int_{\Sl(2,\mathfrak{o})} W_v(g(1,k_2)) \Diff k_2$
over $\id\times \Sl(2,\mathfrak{o}_k)\subseteq H$.

The integral over $t_1$ in the last line equals $L(s,\nu^{1/2}\rho)$ for unramified $\rho$ and is zero otherwise.
The integral over $t_2$ is holomorphic in $s$ because the integrand has compact support with respect to $t_2$.
By choice of $\Phi_1$ we can arrange that the integral over $t_2$ is a non-zero constant.
\end{proof}

\begin{cor}
Every subregular pole attached to $\Pi$ and its split Bessel model with Bessel datum $\Lambda=\rho\boxtimes\rho^\divideontimes$ is of the form
$L(s,\nu^{1/2}\rho\mu)$ or $L(s,\nu^{1/2}\rho^\divideontimes\mu)$.
\end{cor}

It suffices to discuss the subregular poles $L(s,\nu^{1/2}\rho\mu)$. The remaining poles of the form $L(s,\nu^{1/2}\rho^\divideontimes\mu)$ are then obtained by switching $\rho$ and $\rho^\divideontimes$. This does not change the $TS$-module $M$ by \cite[4.19]{RW}.
Without loss of generality, we will set $\mu=1$ from now on by a suitable twist of $\Pi$.

\subsection{Smooth representations}\label{ss:representations}

For a locally compact totally disconnected group $X$ fix the modulus character $\delta_X:X\to\R_{>0}$ with
$$\int_X f(xx_0) d x = \delta_X(x_0) \int_X f(x) d x\ ,\qquad x,x_0\in X\ ,\quad f\in C^\infty_c(X)$$ with respect to a left-invariant Haar measure on $X$.
The abelian category of smooth complex-valued linear representations $\rho$ of $X$ is denoted $\CCC_X$. The full subcategory of representations with finite length is $\CCC_X^{\fin}$.
The \emph{smooth dual} or \emph{contragredient} of $\rho$ is denoted $\rho^\vee$.
For a closed subgroup $Y\subseteq X$,
the \emph{unnormalized compact induction} of $\sigma\in\CCC_Y$ is the representation $\ind_Y^X(\sigma)\in\CCC_X$,
given by the right-regular action of $X$ on the vector space of functions
$f:X\to \C$ that satisfy $f(yxk) = \sigma(y)f(x)$ for every $y\in Y, x\in X$ and $k$ in some open compact subgroup of $X$ and such that $f$ has compact support modulo $Y$.
This defines an exact functor $\ind_Y^X$ from $\CCC_Y$ to $\CCC_X$.
\begin{lem}[Frobenius reciprocity]\label{lem:Frobenius}
For smooth $\rho\in\CCC_X$ and $\sigma\in\CCC_Y$, there is a canonical isomorphism
$$\Hom_X(\ind_Y^X(\sigma),\rho^\vee) \cong \Hom_Y(\sigma,(\rho|_Y)^\vee\otimes\frac{\delta_Y}{\delta_X})\ .$$
\end{lem}
\begin{proof}
See \cite[2.29]{Bernstein-Zelevinsky76}.
\end{proof}

Irreducible smooth representations of $\Gl(1)$ are characters $\chi:k^\times\to\C^\times$. The valuation character $\nu:\Gl(1)\to\C^\times$ is defined by $\nu(\lambda)=\vert \lambda\vert$.
Every $X\in\CCC_{\Gl(1)}^\fin$ of finite length is a direct sum of Jordan blocks $\chi^{(n)}$ with respect to the monodromy operator
$\tau_\chi=(X(\varpi)-\chi(\varpi)\id_X)$ with a uniformizer $\varpi\in\mathfrak{o}_k$.

For smooth representations of $\Gl(2)$ we use the established notation of Bernstein and Zelevinski~\cite{Bernstein-Zelevinsky76, Bernstein-Zelevinsky77}. Fix the standard Borel subgroup $B_{\Gl(2)}$ of upper triangular matrices in $\Gl(2)$, then for smooth characters $\chi_1$ and $\chi_2$ of $\Gl(1)$ we denote normalized parabolically induced representations by
$$\chi_1\times\chi_2=\ind_{B_{\Gl(2)}}^{\Gl(2)}(\chi_1\nu^{1/2}\boxtimes\chi_2\nu^{-1/2})\ .$$
For $\nu^{1/2}\times\nu^{-1/2}$ we also write $M_{(\St:1)}$ since this is the unique indecomposable extension of the trivial representation by the Steinberg representation $\St$.
Analogously, we write $M_{(1:\St)}$ for $\nu^{-1/2}\times\nu^{1/2}$.
Extensions between $\Gl(2)$-modules are classified by the $\Ext^1$-groups, see section~\ref{s:Ext}.

Let $\CCC_{n}=\CCC_{\Gl_a(n)}$ denote the category of smooth representations of $\Gl_a(n)$.
As in \cite{RW} we write $j_!, j^!, i_*,i^*$ for the unnormalized functors $\Phi^+,\Phi^-,\Psi^+,\Psi^-$ of \cite{Bernstein-Zelevinsky76}.
\begin{lem}\label{lem:Gelfand_Kazdhan}
There is a functorial short exact sequence of excision type
\begin{equation*}
 0 \longrightarrow j_!j^!\longrightarrow \id\longrightarrow i_*i^* \longrightarrow 0\ .
\end{equation*}
The composition $j^!i_*$ and $i^*j_!$ are both zero.
\end{lem}
For each smooth character $\rho$ of $\Gl(1)$, the functors $k_\rho, k^\rho\,:$ $\CCC_{n+1}\to \CCC_{n}$ are constructed in \cite{RW} together with a functorial short exact sequence
\begin{equation*}
 0\to j_!k_{\nu\rho} \to k_\rho j_! \to i_*i^*\to 0
\end{equation*}
and a natural equivalence $j_!k^{\nu\rho}\cong k^\rho j_!$.
The projection
\begin{equation*}
Q_1\ni
\begin{pmatrix}ad-bc & sc-ta &\ast&sd-tb\\0&a&s&b\\0&0&1&0\\0&c&t&d\end{pmatrix} \longmapsto \begin{bmatrix}[cc|c]a&b&s\\c&d&t\end{bmatrix}
\in \Gl_a(2)
\end{equation*}
defines a functor $\eta:\CCC_G(\omega)\to\CCC_2$ by taking $S_A$-coinvariant quotients.
We write $\overline{\Pi}=\eta(\Pi)=\Pi_{S_A}\in\CCC_2$ for $\Pi\in\CCC_{G}(\omega)$.
We define the Bessel functors $\CCC_G(\omega)\to\CCC_1$ as $\beta_\rho=k_\rho\circ\eta$ and $\beta^\rho=k^\rho\circ\eta$ .
The Bessel module of an irreducible $\Pi\in\CCC_G(\omega)$ is denoted $\widetilde{\Pi}=\beta_\rho(\Pi)$ if $\rho$ is clear from the context.
The Gelfand-Graev-module is $$\mathbb{S}_n=(j_!)^n(\C)\in\CCC_n\ .$$

\subsection{Sufficient condition}

This condition will be verified in section~\ref{s:Main} in order to show the existence of the subregular poles listed in table~\ref{tab:subregular_poles}.

\begin{lem}\label{lem:suff_cond_subregular_poles}
Fix an irreducible $\Pi\in\CCC_G(\omega)$ and a smooth character $\rho$ of $\Gl(1)$ and let $\widetilde{\Pi}=k_\rho(\overline{\Pi})$ and $K=\Gl(2,\mathfrak{o}_k)$.
If the composition of the canonical morphisms 
\begin{equation*}
\overline{\Pi}^{K}
\to \overline\Pi
\to \widetilde{\Pi}
\to \widetilde{\Pi}/\widetilde{\Pi}^S
\to k_{\nu^2\rho}(\widetilde{\Pi}/\widetilde{\Pi}^S)
\end{equation*}
is non-trivial, then $\rho$ is unramified, the Bessel character $\Lambda=\rho\boxtimes\rho^\divideontimes$ yields a split Bessel model of $\Pi$ and the associated subregular $L$-factor $L^{\PS}_{\sreg}(s,\Pi,\Lambda,1)$ admits a pole of the form
$L(s,\nu^{1/2}\rho)$.
\end{lem}

\begin{proof}
The character $\rho$ has to be unramified since otherwise the composition cannot be non-trivial.
We write $M=\nu^{-3/2}\otimes (\widetilde \Pi/\widetilde\Pi^S)$ as before.
If $\Lambda$ does not yield a Bessel model, then $M$ would be zero.
Since $M$ is perfect, there is an embedding of $TS$-modules (unique up to scalars)
$$M \hookrightarrow C_b^\infty(k^\times)$$
and we identify $M$ with its image \cite[lemma~3.17]{RW}.
This defines the Bessel model of the representation $\Pi$ with respect to the split Bessel datum $\Lambda$.
Lemma~3.31 of \cite{RW} implies that for each $s\in\C$, the functional
$$I_{s}: M\to \CC \quad , \quad f\mapsto \lim_{s'\to s} \frac{Z(s',f,1)}{L(s',M)} 
 $$
spans the one-dimensional space
$%
\Hom_{T}(M, \nu^{-s})\ .$
Fix a pole $s_0$ of the Tate $L$-factor $L(\nu^{1/2}\rho,s)$, 
then $ \nu^{-s_0}= \nu^{1/2}\rho=:\chi_{\crit}$.
In other words, $I_{s_0}$ spans the one-dimensional space
$$ \Hom_{T}(M, \chi_{\crit}) \cong \Hom_{T}(\tilde\Pi/\tilde\Pi^S, \nu^{3/2}\chi_{\crit}) 
\cong \Hom_T( k_\chi(\widetilde{\Pi}/\widetilde{\Pi}^S),\C)\ .$$ 
By our assumption the image of $\overline{\Pi}^K$ in $k_\chi(\widetilde{\Pi}/\widetilde{\Pi}^S)$ is non-trivial.
The map $\Pi^{K} \to \overline\Pi^{K}$ is surjective, because $K$ is compact. 
Hence there is $v\in\Pi^{K}$ such that its image $f_v\in M$ satisfies $I_{s_0}(f_v)\neq0$, or in other words, such that
$$\lim_{s\to s_0}Z^{\PS}_{\reg}(s,W_v,1)/L(s,M)\neq0\ .$$ 
Now proposition~\ref{2.1} implies the last assertion.
\end{proof}

\subsection{Necessary condition}
This condition will be used in section~\ref{s:H_+-functionals} in order to show that there are not more subregular poles than listed in table~\ref{tab:subregular_poles}.
Fix an irreducible smooth representation $\Pi\in\CCC_G(\omega)$ and a smooth character $\rho$ of $\Gl(1,k)$.
We write $\widetilde{\rho}=\nu\rho$ where $\nu(x)=|x|$ for $x\in k^\times$.

\begin{defn}
An $(H_+,\widetilde{\rho})$-\emph{functional} attached to $(\Pi,\rho)$ is a nontrivial functional $\ell:\Pi\to\C$,
such that
$$ \ell(\Pi(s_Azg)v) = \omega(z) \widetilde\rho(\det(g)) \cdot \ell(v) \ $$
for $z\in Z_G$, $s_A\in S_A$ and $g\in \Gl(2)$.
Every $H_+$-functional factorizes over the projection $\Pi\to \overline{\Pi}=\Pi_{S_A}$ and this defines an isomorphism
\begin{equation*}
\{(H_+,\widetilde{\rho})\text{-functionals}\}
\quad \cong \quad \Hom_{\Gl(2)}(\overline{\Pi},\widetilde{\rho}\circ\det)\ .
\end{equation*}
\end{defn}

\begin{prop}\label{prop:poles_give_functionals}
Assume that $\Lambda=\rho\boxtimes{\rho}^\divideontimes$ yields a split Bessel model for $\Pi$.
If $L(\rho\nu^{1/2},s)$ is a subregular pole of $\Pi$
with unramified $\rho$, then $\Pi$ admits a nontrivial
$(H_+,\widetilde{\rho})$-functional $\ell$.
Especially, $\Hom_{\Gl(2)}(\overline{\Pi},\widetilde{\rho}\circ\det)\neq0\,$.
\end{prop}

\begin{proof}
Assume $L(\rho\nu^{1/2},s)$ divides $L^\PS_{\sreg}(s,\Pi,1,\Lambda)$ and admits a subregular pole $s_0\in\C$.
Without loss of generality we can assume that
this pole comes from a pole of a partial zeta integral
$Z^\PS_+(s,W_v,\Phi,\Lambda,1)$
with $\Phi=\Phi_1\boxtimes\Phi_2$ and $\Phi_2(0,0)\neq0$ as in the proof of proposition~\ref{2.1}.
The functional $\ell$ is then defined as the residue functional
$$\ell:\quad v\mapsto \mathrm{res}_{s=s_0}\ Z_+(s,W_v,\Phi,\Lambda,1)/L(s,M)\ .$$
It is non-trivial because the pole has order one.
This functional $\ell$ must factorize over the evaluation $\Phi_2\mapsto\Phi_2(0,0)$.
The partial zeta-integral is $S_A$-invariant because $S_A$ is a normal in $H_+$.
The center clearly acts with the central character $\omega$ of $\Pi$. 
The action of $g\in \Gl(2)\hookrightarrow H_+$ on $f(v)$ is by multiplication with
$$\nu^{-s_0-\frac12}(\det g)\delta_{H_+}(g)=\nu^{-s_0+\frac{1}{2}}(\det g)\ .$$
The pole at $s=s_0$ is a pole of $L(\rho\nu^{1/2},s)$, this means $\rho\nu^{s_0+\frac{1}{2}}=1$.
Hence $g\in \Gl(2)$ acts by multiplication with $\nu\rho(\det g)=\widetilde{\rho}(\det g)$.
\end{proof}

We did not discuss subregular poles of the form $L(\rho^\divideontimes\nu^{1/2},s)$.
In fact, the Weyl involution $\mathbf{s}_1$ switches $\rho$ and $\rho^\divideontimes$.
Hence we can apply our results also for the remaining case of subregular poles.
The subregular $L$-factor $L(\chi_{\crit},s)$ only depends on $\Pi$
and the pair $(\rho, \rho^\divideontimes)$,
where the roles of $\rho$ and $\rho^\divideontimes$ may switch.

\section{Classification of $H_+$-functionals}\label{s:H_+-functionals}

In this section we classify the $(H_+,\widetilde{\rho})$-functionals attached to irreducible smooth representations $\Pi$ of $G$ and smooth characters $\widetilde{\rho}$ of $k^\times$.
In other words, we determine $\dim\Hom_{\Gl(2)}(\overline{\Pi},\widetilde{\rho}\circ\det)$.
The final result is
\begin{thm}\label{MainL}

If $\Pi$ is non-generic, then for every smooth character $\widetilde{\rho}$ of $\Gl(1)$
\begin{gather*}
\dim\Hom_{\Gl(2)}(\overline\Pi, \tilde\rho\circ\det)  =
\begin{cases}
1 & \text{type \nosf{IIb, IVc, Vbc, XIb} and } \widetilde{\rho}\in\Delta_+(\Pi))\ \,,\\
1 & \text{type \nosf{IIIb}, \nosf{VIc}, \nosf{VId}  and } \widetilde{\rho}\in\Delta_+(\Pi)\cup\nu\Delta_+(\Pi)\ ,\\
1 & \text{type \nosf{IVc}  and } \widetilde{\rho}=\nu^2\sigma\in \nu\Delta_+(\Pi)\ ,\\
1 & \text{type \nosf{IVd} and } \widetilde{\rho} = \sigma\ \,,\\
0 & \text{otherwise}\ .
\end{cases}
\end{gather*}

If $\Pi$ is generic,
then for every smooth character $\widetilde{\rho}$ of $\Gl(1)$
\begin{equation*}
\dim\Hom_{\Gl(2)}(\overline\Pi,\widetilde{\rho}\circ\det)=
\begin{cases}
 1 & \text{type \nosf{I, IIa, Va, VIa, X, XIa} and }\widetilde{\rho}\in\Delta_+(\Pi)\ ,\\
 0 & \text{otherwise}\ .
\end{cases}
\end{equation*}
\end{thm}
\begin{proof}
Every $\Gl(2)$-equivariant functional of $\overline{\Pi}$ factorizes over the central specialization.
For non-generic $\Pi$, the central specialization has been calculated explicitly  %
in proposition~\ref{prop:central_specialization_non_generic}.
The assertion follows via a case by case inspection of table \ref{tab:central_specialization}.
For generic $\Pi$, see theorem~\ref{EXGEN}.
\end{proof}
\begin{rmk}
Broadly speaking, $(H_+,\widetilde{\rho})$-functionals exist for $\widetilde{\rho}\in\Delta_+(\Pi)$ and certain other cases which are determined by the Klingen-Jacquet module and satisfy $\widetilde{\rho}\in\nu\Delta_+(\Pi)$.
However, there are exceptions to this rule.
For $\Pi$ of type \nosf{IIIa}, \nosf{IVa} and \nosf{IVb} there is no  functional for $\widetilde{\rho}\in\Delta_+(\Pi)$.
For $\Pi$ of type \nosf{IVc} there is no functional for $\widetilde{\rho}=\sigma\in\nu\Delta_+(\Pi)$.
The one-dimensional representations $\Pi$ of type \nosf{IVd} admit  functionals for $\widetilde{\rho}=\sigma$ although $\nu\Delta_+(\Pi)$ is empty.
\end{rmk}

\begin{cor}\label{cor:functionals_with_Bessel_models}
Suppose $\Lambda=\rho\boxtimes\rho^\divideontimes$ defines a split Bessel model for an irreducible representation $\Pi$ of $G$, then the $\widetilde{\rho}$-equivariant functionals with $\widetilde{\rho}=\nu\rho$ are 
\begin{gather*}
\dim\Hom_{\Gl(2)}(\overline\Pi, \tilde\rho\circ\det)  =
\begin{cases}
 1 & \text{type \nosf{I, IIa, Va, VIa, X, XIa} and }\widetilde{\rho}\in\Delta_+(\Pi)\\
1 & \text{type \nosf{IIIb} and } \widetilde{\rho}\in\nu\Delta_+(\Pi)=\{\nu\sigma,\nu\chi_1\sigma\}\ ,\\
1 & \text{type \nosf{IVc}  and } \widetilde{\rho}=\nu^2\sigma\in \nu\Delta_+(\Pi)\ ,\\
1 & \text{type \nosf{VIc}  and } \widetilde{\rho}\in\nu\Delta_+(\Pi)=\{\nu\sigma\}\ ,\\
1 & \text{type \nosf{VId}  and } \widetilde{\rho}\in\nu\Delta_+(\Pi)=\{\nu\sigma\}\ ,\\
0 & \text{otherwise}\ .
\end{cases}
\end{gather*}

These functionals factorize over the projection $\overline\Pi\to i^*(\overline\Pi)=B$ if and only if $\Pi$ is not generic.

Especially, there are not more subregular poles than the ones listed in table~\ref{tab:subregular_poles}.
\end{cor}
\begin{proof}
 Recall that $\Lambda=\rho\boxtimes\rho^\divideontimes$ provides a split Bessel model for $\Pi$ if and only if $\Pi$ is generic or if $\rho\in\Delta_+(\Pi)$.
 Compare \cite[thm.~6.2]{RW} and \cite[thm.\ 6.2.2]{Roberts-Schmidt_Bessel}.
 The last assertion follows by prop.~\ref{prop:poles_give_functionals}.
\end{proof}

\begin{rmk}
Type \nosf{VIc} and \nosf{VId} are extended Saito-Kurokawa cases.
Here the functionals have already been constructed in \cite[lemma~5.4]{W_Excep}
and both cases yield subregular poles of the form $L(s,\nu^{1/2}\sigma)$.
\end{rmk}

In the following part of this section we prove theorem~\ref{MainL}.

\begin{lem}\label{lem:H_+_functionals}
For every irreducible $\Pi\in\CCC_G(\omega)$ and every smooth character $\widetilde{\rho}$ of $\Gl(1)$,
there is an embedding
\begin{equation*}
\Hom_{\Gl(2)}(\overline\Pi, \tilde\rho\circ\det) \hookrightarrow \Hom_{T}(k_{\tilde\rho}(\overline{\Pi}), \tilde\rho) \cong  \Hom_{\C}(k_{\widetilde{\rho}}k_{\tilde\rho}(\overline{\Pi}), \C) \ .
\end{equation*}
\end{lem}

\begin{proof}
Every $\Gl(2)$-equivariant map $   \overline \Pi \to  (\C,\tilde\rho\circ\det) $
is $B_{\Gl(2)}$-equivariant for the standard Borel subgroup $B_{\Gl(2)}$ of upper triangular matrices in $\Gl(2)$.
It thus factorizes over a unique $T$-linear map
$  k_{\tilde\rho}(\overline\Pi)  \to    (\C,\tilde\rho) \ $
and hence a unique map
$k_{\widetilde{\rho}}k_{\widetilde{\rho}}(\overline{\Pi})\cong k_{\widetilde\rho}k_{\widetilde\rho}(\overline\Pi)\to \C$.
\end{proof}

In the following we tacitly assume that all $\Pi$ are normalized as \cite[table~1]{RW}.
Recall $\Delta_+(\Pi)\cap \nu\Delta_+(\Pi)\neq \emptyset$ if and only if $\Delta_-(\Pi)\cap \Delta_+(\Pi)\neq \emptyset$ if and only if
$\Pi$ is of type \nosf{IIIab} with $\chi_1=\nu^{\pm 1}$ by Lemma A.11 in \cite{RW}.
Hence, for non-generic $\Pi$ the characters $\rho$ and $\tilde\rho$
both define a Bessel model only in these cases.
For every other non-generic $\Pi$ with $\rho\in \Delta_+(\Pi)$, the character $\tilde\rho$ is not in $\Delta_+(\Pi)$.
\begin{lem}\label{key}
Suppose $\rho$ is an arbitrary smooth character of $k^\times$ that
does not define a split Bessel model for a smooth irreducible representation~$\Pi$.
Then $\Pi$ is necessarily non-generic.
The Bessel module $\beta_{\rho}(\Pi)=k_\rho(\overline{\Pi})$ is given by
$$\beta_{\rho}(\Pi) \cong \beta^{{\rho}}(\Pi)\oplus\bigoplus_{\chi\in \Delta_0(\Pi)} \nu^{3/2}\chi \ . $$
Every $T$-character occurs at most once, so $\beta_{\rho}(\Pi)$ is semisimple.
\end{lem}

\begin{proof}
See lemma~5.25 and table 7 in \cite{RW}. 
For the list of constituents, see \cite[table 3, prop.~6.3.3]{RW}.
\end{proof}

\begin{lem}\label{Nice}
For every irreducible $\Pi\in\CCC_G(\omega_\Pi)$ and smooth characters $\rho$ we have $\dim\beta_{\widetilde\rho}(\Pi)_{T,\tilde\rho}\leq 1$.
Equality holds if $\rho$ provides a Bessel model for $\Pi$. %
\end{lem}
\begin{proof} Suppose $\widetilde{\rho}$ defines a Bessel model for $\Pi$, i.e. $\deg(\beta_{\widetilde\rho}(\Pi)=1$.
Then $\dim\beta_{\tilde\rho}(\Pi)_{T,\chi}=1+\dim\beta_{\tilde\rho}(\Pi)^{T,\chi}$ holds for every smooth character $\chi$ \cite[lemma~3.12]{RW}.
It has been shown that $\dim\beta_{\tilde\rho}(\Pi)^{T,\chi}$ vanishes for every smooth character $\chi\neq \nu^2\widetilde{\rho}$ \cite[theorem~6.6]{RW}.
If $\tilde\rho$ does not define a Bessel model for $\Pi$, then
 $\dim\beta_{\tilde\rho}(\Pi)_{T,\tilde\rho}\leq1$ by lemma~\ref{key} because every $T$-character occurs at most once in $\beta_{\widetilde{\rho}}(\Pi)$.
\end{proof}

\begin{lem}\label{tech}\label{mult}\label{mult_P3}
Fix irreducible representations $\Pi\in\CCC_G$ and $\pi\in\CCC_{\Gl(2)}$.
If $\pi$ is one-dimensional or if $\Pi$ is non-generic, then
 \begin{equation*}
 \dim\Hom_{\Gl(2)}(\overline{\Pi},\pi)\leq1\ .
 \end{equation*}
\end{lem}
\begin{proof}
We first assume that $\pi=(\widetilde\rho\circ\det)$ with some character $\rho\in\CCC_{\Gl(1)}\,.$
By dual Frobenius reciprocity,
$\Hom_{\Gl(2)}(\overline{\Pi}, \tilde\rho\circ\det)$ is contained in
$$ \Hom_{\Gl(2)}(\overline\Pi, \tilde\rho\otimes M_{(1:\St)})
\cong \Hom_{B_{\Gl(2)}}(\overline\Pi, \tilde\rho\boxtimes \tilde\rho) \cong \Hom_T(k_{\widetilde{\rho}}(\overline\Pi), \tilde\rho) \ .$$
By lemma \ref{Nice} the right hand side has dimension $\leq 1$.
Now assume that $\Pi$ is non-generic. If $\pi$ is cuspidal, a non-zero morphism $\overline{\Pi}\to \sigma$ can only exist for type \nosf{VIIIb} and \nosf{IXb} by lemma~\ref{lem:AB_nongeneric}. In both cases Frobenius reciprocity and lemma~\ref{lem:AB-sequence} imply the assertion.
It suffices to show the assertion for fully parabolically induced representations
$\pi=\ind^{\Gl(2)}_{B_{\Gl(2)}}(\chi\boxtimes\rho)$ with arbitrary smooth characters $\chi$ and $\rho$.
The argument is analogous to the first case and uses $\dim\beta_\rho(\Pi)_{T,\chi}\leq1\,$ by \cite[cor.~5.23]{RW}.
\end{proof}

Lemma~\ref{mult} probably holds for arbitrary irreducible $\Pi$ and $\pi$.
For generic $\Pi$ and non-cuspidal $\pi$ this follows by the same proof used in proposition~\ref{EXGEN}.
If $\Pi$ and $\pi$ are both generic, we can at least show that $\dim\Hom_{\Gl(2)}(\overline{\Pi},\pi)$ is non-zero but not larger than two, see corollary~\ref{cor:central_specialization_generic}.

\subsection{Gelfand-Kazhdan theory}\label{ssec:Gelf_Kazhdan}
For an irreducible ${\Pi}\in\CCC_G(\omega)$,
fix the unnormalized Jacquet modules $A=J_P(\Pi)_\psi\cong {i^*j^!(\overline{\Pi})}\in\CCC_{\Gl(1)}$ and $B=J_Q(\Pi)\cong i^*(\overline{\Pi})\in\CCC_{\Gl(2)}$. Let ${m_\Pi}=0,1$ be the dimension of Whittaker models of $\Pi$.
\begin{lem}\label{lem:AB-sequence}
For every $\Pi\in\CCC_G(\omega)$ there is a short exact sequence of $\Gl_a(2)$-modules
$$0\to j_!i_*(A)\to \overline{\Pi}/\mathbb{S}_2^{m_\Pi}\to i_*(B)\to0\ .$$
\end{lem}
This is the $P_3$-filtration of Gelfand and Kazhdan, compare \cite[\S4.2]{RW}.

\begin{lem}\label{lem:H_+functional_decomposition}
For every $\Pi$ in $\CCC_G(\omega)$ there is an isomorphism
$$ \Hom_{\Gl(2)}(\overline{\Pi}, \tilde\rho\circ\det) \cong 
\Hom_{\Gl(2)}(\overline{\Pi}/\mathbb{S}_2^{m_\Pi}, \tilde\rho\circ \det) \ .$$
Further, there is an exact sequence
$$  0\to \Hom_{\Gl(2)}(B,\tilde\rho\!\circ\!\det) \to \Hom_{\Gl(2)}(\overline{\Pi}/\mathbb{S}_2^{m_\Pi}, \tilde\rho\!\circ\!\det)  \to \ker(\delta) \to 0 $$
with boundary map
$  \delta:  \Hom_{\Gl(2)}(j_!i_*(A), \tilde\rho\circ\det) \to \Ext^1_{\Gl(2)}(B, \tilde\rho\circ\det)
\ .$
\end{lem}
\begin{proof}
Lemma~2.2 in \cite{W_Excep} shows $\Hom_{\Gl(2)}(\mathbb{S}_2,\widetilde{\rho}\circ\det)=0$.
By lemma~\ref{lem:AB-sequence}, the $(H_+,\widetilde{\rho})$-functionals of $\Pi$ come from those of $A$ or $B$.
Now both assertions follow from right-exactness of the functor $\Hom_{\Gl(2)}(-,\widetilde{\rho}\circ\det)$.
\end{proof}

\begin{lem}\label{H_+functionals_for_B}
For irreducible $\Pi\in\CCC_G(\omega)$ and every smooth character $\widetilde{\rho}$ of $k^\times$,
\begin{equation*}
\dim\Hom_{\Gl(2)}(B,\tilde\rho\circ\det)=
\begin{cases}
 1 & \text{ type \nosf{IIIb}}\text{ and }\widetilde{\rho}\in\{\nu\sigma,\chi_1\nu\sigma\},\\
 1 & \text{ type \nosf{IVc}} \text{ and }\widetilde{\rho}=\nu^2\sigma,\\
 1 & \text{ type \nosf{IVd}} \text{ and }\widetilde{\rho}=\sigma,\\
 1 & \text{ type \nosf{VIc}} \text{ and }\widetilde{\rho}=\nu\sigma,\\
 1 & \text{ type \nosf{VId}} \text{ and }\widetilde{\rho}=\nu\sigma,\\
 0 & \text{ otherwise}\ .
\end{cases}
\end{equation*}
\end{lem}
\begin{proof}For generic $\Pi$ the Klingen-Jacquet-module has no one-dimensional constituents~\cite[table A.4]{Roberts-Schmidt}. For non-generic $\Pi$ see lemma~\ref{lem:AB_nongeneric}.
The dimension cannot exceed one by lemma~\ref{mult}.
\end{proof}

\begin{lem}\label{lem:A_cyclic}
For every irreducible $\Pi\in\CCC_G(\omega)$, the $T$-module $A=i^*j^!(\overline{\Pi})$ is cyclic. In other words, every $T$-character occurs with at most a single Jordan block in $A$.
\end{lem}
\begin{proof}
 If $\Pi$ is generic, then the $\Gl_a(1)$-module $j^!(\overline{\Pi})$ is perfect of degree one \cite[4.14]{RW}, so $A$ is cyclic by \cite[3.17]{RW}.
 If $\Pi$ is non-generic, then the constituents of $A$ are pairwise distinct, see \cite[table A.3]{RW}.
\end{proof}

\begin{lem}[{\cite[lemma~2.2.2]{W_Excep}}]\label{H_+functionals_for_A}
For irreducible $\Pi\in\CCC_G(\omega)$ and every smooth character $\widetilde{\rho}$ of $k^\times$,
\begin{equation*}
\dim\Hom_{\Gl(2)}(j_!i_*(A)), \tilde\rho\circ\det)=
\begin{cases}
1 & \widetilde{\rho}  \in \Delta_+(\Pi)\ ,\\
0 & \text{ otherwise }\ .\end{cases}
\end{equation*}
\end{lem}

\begin{proof}
Every functional factorizes over the central specialization given by lemma~\ref{lem:cent_loc_j_!i_*}.
For every Jordan block $\chi^{(n)}$ in $A$,
Frobenius reciprocity implies
$$\Hom_{\Gl(2)}( j_! i_*(\chi^{(n)}), \tilde\rho\circ \det)
\cong
\Hom_{\Gl(1)}(\chi^{(n)},\widetilde{\rho}\nu)\ .
$$
This space has dimension one for %
$\widetilde{\rho}=\nu^{-1}\chi\in\Delta_+(\Pi)$ and is zero otherwise.
The assertion follows by lemma~\ref{lem:A_cyclic}.
\end{proof}

\begin{table}
\begin{footnotesize}
\caption{Modules $A$ and $B$ for non-generic $\Pi$. \label{tab:AB}}
\begin{tabular}{llll}
\toprule
Type        & $\Pi\in\CCC_G$          & $A\in\CCC_{T}$ & $B\in\CCC_{\Gl(2)}$ %
\\
\midrule
\nosf{IIb}  & $(\chi\circ\det)\rtimes\sigma$, $\chi^2\neq1$ & $\nu\chi\sigma $ & $\sigma(\nu^{1/2}\chi\times\nu\chi^2)\oplus \sigma(\nu^{1/2}\chi\times\nu)$ \\
\nosf{IIb}  & $(\chi\circ\det)\rtimes\sigma$, $\chi^2=1$ & $\nu\chi\sigma $ & $\sigma (\nu^{1/2}\chi\times\nu^{(2)})$ \\
\nosf{IIIb} & $\chi_1\rtimes(\sigma\circ\det)$, $\chi_1 \neq\nu^{\pm1}$ & $\nu\sigma\oplus\chi_1\nu\sigma$ & $(\nu\chi_1\sigma\circ\det)\oplus (\nu\sigma\circ\det) \oplus \nu^{1/2}\sigma(\chi_1\times1)$ \\
\nosf{IIIb} & $\nu\rtimes(\sigma\circ\det)$ &$\nu\sigma\oplus\nu^2\sigma$ & $(\nu^2\sigma\circ\det)\oplus \nu\sigma M_{(1:\St:1)}^{nc}$\\
\nosf{IVb}  & $L(\nu^2,\nu^{-1}\sigma\St)$ &$\nu\sigma$&$\sigma\St\oplus\sigma(\nu^{5/2}\times\nu^{1/2})$ \\
\nosf{IVc}  & $L(\nu^{3/2}\St,\nu^{-3/2}\sigma)$ &$\sigma\oplus\nu^2\sigma$&$(\nu^2\sigma\circ\det)\oplus \sigma(\nu^{3/2}\times\nu^{-1/2})$ \\
\nosf{IVd}  & $\sigma\circ\lambda$         &$0$&$(\sigma\circ\det)$\\
\nosf{Vb}   & $L(\nu^{1/2}\xi\St,\nu^{-1/2}\sigma)$ &$\nu\sigma$&$\nu^{1/2}\sigma(\nu\xi\times1)$\\
\nosf{Vd}   & $L(\nu\xi,\xi\rtimes\nu^{-1/2}\sigma)$&$0$&$\nu\sigma(\xi\times1)$\\
\nosf{VIb}  & $\tau(T,\nu^{-1/2}\sigma)$        &$0$    &$\nu\sigma\St$            \\

\nosf{VIc}  & $L(\nu^{1/2}\St,\nu^{-1/2}\sigma)$&$\nu\sigma$&$(\nu\sigma\circ\det)$\\
\nosf{VId}  & $L(\nu,1\rtimes\nu^{-1/2}\sigma)$ &$\nu\sigma$&$(\nu\sigma\circ\det)\oplus \nu^{1/2}\sigma(1\times1)$\\
\nosf{VIIIb}& $\tau(T,\pi_c)$                   &$0$&$\nu\pi_c$         \\
\nosf{IXb}  &$L(\nu\xi,\nu^{-1/2}\pi_c)$        &$0$&$\nu^{1/2}\xi\pi_c$\\
\nosf{XIb}  &$L(\nu^{1/2}\pi_c,\nu^{-1/2}\sigma)$&$\nu\sigma$&$0$       \\
\bottomrule
\end{tabular}

\end{footnotesize}
\end{table}

\subsection{Non-generic cases}

For the non-generic $\Pi\in\CCC_G(\omega)$ we completely determine the central specialization $\zeta_\mu(\overline{\Pi}$. It is then easy to determine the $(H_+,\widetilde{\rho})$-functionals of $\Pi$ because they all factorize over $\zeta_\mu(\overline{\Pi})$ with $\mu=\widetilde{\rho}^2$.

\begin{lem}\label{lem:AB_nongeneric}
 For non-generic irreducible representations $\Pi$ of $G$ the modules $A$ and $B$ are explicitly given by table~\ref{tab:AB}, where $M^{nc}_{(1:\St:1)}$ is uniquely determined as the kernel of the projection $\nu^{-1/2}\times(\nu^{1/2})^{(2)}\twoheadrightarrow\St$.
\end{lem}
The $\Gl(2)$-module $M^{nc}_{(1:\St:1)}$ is indecomposable of length three without a central character.
Its socle and top are both the trivial representation and the third constituent is the Steinberg representation.
\begin{proof}
The constituents of $B\in\CCC_{\Gl(2)}$ are isomorphic to $(\nu\chi'\circ\det)\otimes\tau$ for the $\Gl(1)\times\Gl(2)$-modules $\chi'\otimes\tau$ that occur in table~A.4 of \cite{Roberts-Schmidt}.
The constituents of $A$ are given by $\nu^{3/2}\Delta_0(\Pi)$, see table~3 in \cite{RW}.
Non-trivial extensions can only exist between the constituents of $B$ for the representation $\chi_1\rtimes(\sigma\circ\det)$ in case \nosf{IIIb} where $\chi_1=\nu^{\pm1}$ and for the representation $\Pi=(\chi\circ\det)\rtimes\sigma$ of type \nosf{IIb} with $\chi^2=1$.

Suppose $\Pi=\chi_1\rtimes(\sigma\circ\det)$ is of type \nosf{IIIb} with $\chi_1=\nu^{\pm1}$, then we can assume $\chi_1=\nu$ and $\sigma=1$ without loss of generality by a Weyl reflection and a twist. Indeed, the Weyl involution $\mathbf{s}_1\mathbf{s}_2\mathbf{s}_1$ defines an isomorphism $\nu^{-1}\rtimes(\sigma\circ\det)\cong \nu\rtimes(\nu^{-1}\sigma\circ\det)$.
By table~A.4 in \cite{Roberts-Schmidt},
the normalized Klingen-Jacquet-module $J_Q(\Pi)$ has length four and admits a constituent $\nu^{-1}\boxtimes\St(\nu)$, but neither as a submodule nor as a quotient by Frobenius reciprocity.
In other words, the $\Gl(2)$-module $B$ contains the constituent $\St(\nu)$, but neither as a submodule nor as a quotient, and is thus isomorphic to
$$B\cong(\nu^2\circ\det)\oplus ((\nu\circ\det)\otimes M)\ ,$$
where $M$ is an indecomposable $\Gl(2)$-module whose socle and top are both the trivial representation.
This is an extension panach\'{e}e in the sense of \cite[9.3]{SGAVII}.
Central specialization at $\mu=\nu^2\sigma^2$ of the short exact sequence of lemma~\ref{lem:AB-sequence} yields an long exact sequence
\begin{equation*}
 \zeta^\mu(\overline{\Pi})\to \zeta^\mu(i_*(B))\to \zeta_\mu(j_!i_*(A))\to \zeta_\mu(\overline{\Pi})\to \zeta_\mu(i_*(B))\to 0\ .
\end{equation*}
The left-most term $\zeta^\mu(\overline{\Pi})=0$ vanishes \cite[lemma~6.1]{W_Excep}, so $\zeta^\mu(i_*(B))$ is an indecomposable submodule of $\zeta_\mu(j_!i_*(A))$ and thus has length at most two by lemma~\ref{lem:cent_loc_j_!i_*}.
Hence $M=M^{nc}_{(1:\St:1)}$ does not admit a central character.
The exact sequence $(\ast)$ of \cite[p.\ 34f]{W_Excep} associated with the functor $\zeta_{\mu^{(2)}}$ for $\mu=\nu^2\sigma^2$ yields an analogous long exact sequence and thus an embedding
$B=\zeta^{\mu^{(2)}}(B)\hookrightarrow \zeta_{\mu^{(2)}}(j_!i_*(A))$.
By an argument analogous to the proof of lemma~\ref{lem:cent_loc_j_!i_*} there is an isomorphism
$$\zeta_{\mu^{(2)}}(j_!i_*(A)) \cong \sigma\nu(\nu^{-1/2}\times (\nu^{1/2})^{(2)}) \oplus \sigma\nu(\nu^{1/2}\times (\nu^{-1/2})^{(2)})\ .$$
Up to scalars, there is only one embedding of $B\hookrightarrow\zeta_{\mu^{(2)}}(j_!i_*(A))$ and this characterises $B$ uniquely.

Now suppose $\Pi=(\chi\circ\det)\rtimes\sigma$ is type \nosf{IIb} with $\chi_1^2=1$.
Then $B$ is an extension of $\sigma(\nu^{1/2}\chi\times\nu)$ by itself and thus has length two, so $\zeta_{\mu^{(2)}}(B)=B$ for $\mu=\nu^{3/2}\chi\sigma^2$.
The rest of the argument is analogous to the previous case, but it turns out that the embedding
$B\hookrightarrow \zeta_{\mu^{(2)}}(j_!i_*(\nu\chi\sigma))\cong \sigma(\nu^{1/2}\chi\times\nu^{(2)})$
is actually an isomorphism by comparing the lengths.
\end{proof}

\begin{prop}\label{prop:central_specialization_non_generic}
For non-generic irreducible $\Pi\in \CCC_{G}(\omega)$ and every smooth character $\mu$ of $\Gl(1)$, the central specialization $\widehat{\Pi}=\zeta_\mu(\overline{\Pi})$ has finite length and is given by table~\ref{tab:central_specialization}.
\end{prop}
Here the $\Gl(2)$-module $M^c_{(\chi\times1:\chi\times1)}$ (resp.\ $M^{c}_{(\St:1:\St:1)}$) is the unique indecomposible extension of $\chi\times1$ (resp.\ $M_{(\St:1)}$) with itself that admits a central character.

\begin{proof}
Fix the exact sequence $0\to j_!i_*(A)\to\overline{\Pi}\to i_*(B)\to0$ of lemma~\ref{lem:AB-sequence}, where $A$ and $B$ are given in lemma~\ref{lem:AB_nongeneric}.
For $A=0$ the assertion is clear.
For $A\neq0$ lemma~6.1 of \cite{W_Excep} implies $\zeta^\mu(\overline{\Pi})=0$, so there is an exact sequence
\begin{equation}\label{eq:specialization_seq}
0\to\zeta^\mu(B)\to \zeta_\mu(j_!i_*(A))\to\zeta_\mu(\overline\Pi)\to \zeta_\mu(B)\to0\ .\tag{$\ast$}
\end{equation}
Especially, $\widehat{\Pi}=\zeta_\mu(\overline{\Pi})$ has the same constitutents as $\zeta_\mu(j_!i_*(A))$, which are all non-cuspidal and explicitly given by lemma~\ref{lem:cent_loc_j_!i_*}.
Furthermore, if $\zeta_\mu(B)$ vanishes, then so does $\zeta^\mu(B)$ by lemma~\ref{list}, there is an obvious isomorphism $\zeta_\mu(j_!i_*(A))\cong\zeta_\mu(\Pi)$.

It remains to determine the monodromy.
Lemma~\ref{mult_P3} implies that the maximal semisimple quotient of $\zeta_\mu(\overline{\Pi})$ contains every constituent at most once.
We first assume that neither is $\Pi$ of type \nosf{IIIb} with $\mu=\nu\chi\sigma^2$ nor is $\Pi$ of type \nosf{IVc} with $\mu=\nu\sigma^2$.
Lemma~\ref{lem:cent_loc_j_!i_*} implies
\begin{equation*}
\zeta_\mu(j_!i_*(A))\cong\bigoplus_{\chi}(\chi\nu^{-1/2}\times \chi^{-1}\nu^{1/2}\mu) ,\\
\end{equation*}
where the sum runs over constituents $\chi$ of $A$.
By Bernstein decomposition, the $\Gl(2)$-module $\overline{\Pi}$ splits as a direct sum of smooth $\Gl(2)$-modules $Q$ that are either isomorphic to $j_!i_*(\chi)$ for some $\chi$ in $A$ or a constituent of $i_*(B)$ or to extensions of $\Gl(2)$-modules
\begin{equation*}
0\to j_!i_*(\chi) \to Q \to (\widetilde{\rho}\circ\det)\to0\ %
\end{equation*} or
\begin{equation*}
0\to j_!i_*(\chi) \to Q \to \nu^{-1}\chi\otimes \St\to0\ .
\end{equation*}
Lemma~\ref{lem:cent_spez_indecomposable} determines $\zeta_\mu(Q)$ uniquely in each case.

Among the remaining cases, we focus on type \nosf{IIIb} with $\chi_1=\nu$ and $\mu=\nu^2\sigma^2$.
The others are analogous.
Then we have $B=\nu\sigma M^{nc}_{(1:\St:1)}\oplus (\nu^2\sigma\circ\det)$ and $A=\nu\sigma\oplus\nu^2\sigma$, where $M^{nc}_{(1:\St:1)}$ is an extension of length three that does not admit a central character.
We obtain isomorphisms $\zeta_\mu(i_*(B))\cong \nu\sigma M_{(\St:1)}$ and $\zeta^\mu(i_*(B))\cong \nu\sigma M_{(1:\St)}$ and $\zeta_\mu(j_!i_*(A))\cong \sigma\nu M_{(1:\St)}\oplus \sigma\nu M_{(\St:1)}$ and thus \eqref{eq:specialization_seq} becomes
$$0\to\nu\sigma M_{(1:\St)}\to \sigma\nu M_{(1:\St)}\oplus \sigma\nu M_{(\St:1)}\to\zeta_\mu(\overline\Pi)\to \nu\sigma M_{(\St:1)}\to0\ .$$
Hence $\zeta_\mu(\overline{\Pi})$ is an extension of $\nu\sigma M_{(\St:1)}$ by itself
in the category of smooth $\Gl(2)$-modules with central character $\mu$.
By lemma~\ref{mult_P3}, this extension does not split as a direct sum.
Lemma~\ref{lem:Dual_Frob_Ext} implies $$\Ext^1_{\mathrm{PGl}(2)}(M_{(\St:1)},M_{(\St:1)})=1\ ,$$ so there is only one such indecomposable extension up to isomorphism.
\end{proof}

\begin{cor} \label{B}\label{A_vanishes_generic}
Suppose $\Pi$ is non-generic and $\Lambda=\rho\boxtimes\rho^{\divideontimes}$ provides a split Bessel model for $\Pi$.
Then the canonical monomorphism
\begin{gather*}
 \Hom_{\Gl(2)}(B,\tilde\rho\circ\det) \to  \Hom_{\Gl(2)}(\overline\Pi, \tilde\rho\circ\det)
\end{gather*} is an isomorphism and the functionals $\overline{\Pi}\to (\widetilde{\rho}\circ\det)$ are given by lemma~\ref{H_+functionals_for_B}.
Further, $\Hom_{\Gl(2)}(j_!i_*(A), \tilde\rho\circ\det)=0$ vanishes,
unless $\Pi$ is of type \nosf{IIIb} with $\chi_1=\nu^{\pm1}$.
\end{cor}

\begin{proof}A Bessel model exists if and only if $\rho\in\Delta_+(\Pi)$ \cite[thm.\ 6.2.2]{Roberts-Schmidt_Bessel}, \cite[thm.~6.2]{RW}.
For the first assertion, note that the functionals all
factorize over the central specialization at $\mu=\widetilde{\rho}^2$.
Thus the assertion is a direct consequence of lemma~\ref{lem:AB_nongeneric} and proposition~\ref{prop:central_specialization_non_generic}.
The second assertion follows from lemma~\ref{H_+functionals_for_A},
because the intersection of $\Delta_+(\Pi)$ with $\nu\Delta_+(\Pi)$ is empty by \cite[lemma~A.8]{RW} unless $\Pi$ belongs to type \nosf{IIIb} with $\chi_1=\nu^{\pm 1}$.
\end{proof}

\subsection{Generic cases}
For generic $\Pi$ we show that $(H_+,\widetilde{\rho})$-functionals exist exactly for the generic exceptional cases discussed in \cite{RW}.
\begin{prop} \label{EXGEN}
Fix a generic irreducible representation $\Pi\in\CCC_G(\omega)$ and
a smooth character $\widetilde{\rho}$ of ${\Gl(1)}$ and let $\widetilde{\rho}=\nu\rho$,
then
\begin{equation*}
\dim\Hom_{\Gl(2)}(\overline\Pi,\widetilde{\rho}\circ\det)=
\begin{cases}
 1 & \text{type \nosf{I, IIa, Va, VIa, X, XIa} and }\rho\in\Delta_-(\Pi) \ ,\\
 0 & \text{otherwise}\ .
\end{cases}
\end{equation*}
These functionals do not factorize over $B$.
\end{prop}

\begin{proof}
Every $\widetilde{\rho}$-equivariant functional factorizes over the central specialization $\widehat{\Pi}=\zeta_{\mu}(\Pi)\in\CCC=\CCC_{\Gl(2)}(\mu)$
with central character $\mu=\widetilde{\rho}^2$.
The indecomposable extensions 
\begin{equation*}
 0 \to (\widetilde{\rho}\circ\det)\to \pi_+ \to \St(\widetilde{\rho})\to 0\ ,
\end{equation*}
\begin{equation*}
 0 \to \St(\widetilde{\rho}) \to \pi_- \to (\widetilde{\rho}\circ\det)\to 0\ .
\end{equation*}
given by $\pi_+%
=\widetilde{\rho}\otimes M_{(1:\St)}$
and $\pi_-
=\widetilde{\rho}\otimes M_{(\St:1)}$
yield long exact sequences
\begin{equation*} 
\xymatrix{
0\ar[r] & \Hom_{\cal C}(\widehat{\Pi}, \tilde\rho\circ\det)  \ar[r] & \Hom_{\cal C}(\widehat{\Pi}, \pi_+)  \ar[r] & \Hom_{\cal C}(\widehat{\Pi}, \St(\tilde\rho)) \ar@{->}`r[d]`d[l]`^d[lll]_\delta`d[l] [dll] \\
 & \Ext^1_{\cal C}(\widehat{\Pi}, \tilde\rho\circ\det) \ar[r] & \Ext^1_{\cal C}(\widehat{\Pi}, \pi_+)  \ar[r] &  \Ext^1_{\cal C}(\widehat{\Pi}, \St(\tilde\rho))  \ar[r] &  0\ ,
}
\end{equation*}
\begin{equation*}
\xymatrix{
 0 \ar[r] & \Hom_{\cal C}(\widehat{\Pi}, \St(\tilde\rho)) \ar[r] & \Hom_{\cal C}(\widehat{\Pi}, \pi_-)  \ar[r] & \Hom_{\cal C}(\widehat{\Pi}, \tilde\rho\circ\det )  \ar@{->}`r[d]`d[l]`^d[lll]_\delta`d[l] [dll] \\
      & \Ext^1_{\cal C}(\widehat{\Pi}, \St(\tilde\rho))  \ar[r] & \Ext^1_{\cal C}(\widehat{\Pi}, \pi_-)  \ar[r] & \Ext^1_{\cal C}(\widehat{\Pi}, \tilde\rho\circ\det) \ar[r] & 0 \ .
}
\end{equation*}

Suppose $\Pi$ is of type \nosf{I}, \nosf{IIa}, \nosf{Va}, \nosf{VIa}, \nosf{X} or \nosf{XIa} and $\rho\in\Delta_-(\Pi)$.
Then the Bessel module $k_{\rho}(\overline\Pi)$ is non-perfect \cite[\S6]{RW}, so
$\dim \Hom_{\CCC}(\widehat\Pi, \pi_+)=1$ and $\dim \Hom_{\CCC}(\widehat\Pi, \pi_-)=2$
by lemma~\ref{ext-lemma} and lemma~\ref{Nice}.
If $\Hom_{\CCC}(\widehat\Pi, \tilde\rho\circ\det)=0$ vanishes, then by proposition~\ref{Finally} $\Ext_{\CCC}(\widehat\Pi, \tilde\rho\circ\det)=0$ also vanishes.
The first long exact sequence
implies $\dim\Hom_{\CCC}(\widehat\Pi, \St(\tilde\rho))=1$ by proposition~\ref{Finally}.
The second long exact sequence implies $\Hom_{\CCC}(\widehat\Pi, \St(\tilde\rho))=2$.
This is a contradiction, so $\Hom_{\CCC}(\widehat\Pi, \tilde\rho\circ\det)$ is non-zero and thus one by lemma~\ref{mult}.

For $\Pi$ of type \nosf{IIIa}, \nosf{IVa} the Bessel module $k_{\rho}(\overline\Pi)$ is perfect for every $\rho$.
In this case lemma~\ref{ext-lemma} implies $\dim \Hom_\CCC(\widehat\Pi, Y_\pm)=1$.
By proposition~\ref{Euler}, $\Ext_\CCC(\widehat{\Pi},Y_\pm)=0$ vanishes.
The long exact sequences imply
$\Ext_\CCC(\widehat{\Pi},\widetilde{\rho}\circ\det)=0$
and $\Ext_\CCC(\widehat{\Pi},\St(\widetilde{\rho}))=0$ .
Counting dimensions shows
$$\dim \Hom_\CCC(\widehat\Pi, \St(\tilde\rho))+\dim \Hom_\CCC(\widehat\Pi, \tilde\rho\circ\det)=1\ .$$
For $\rho\in\Delta_-(\Pi)$ there are nontrivial $\Gl(2)$-homomorphisms $\overline{\Pi}\twoheadrightarrow B\twoheadrightarrow  \St(\tilde\rho)$ \cite[table~A.4]{Roberts-Schmidt},
so the assertion follows.

The case $\rho\notin\Delta_-(\Pi)%
$
follows by lemma~\ref{lem:H_+functional_decomposition}, lemma~\ref{H_+functionals_for_B} and lemma~\ref{H_+functionals_for_A}.

Finally, these functionals cannot factorizes over $B$ since the Klingen-module of a generic representation $\Pi$ does not admit one-dimensional quotients.
\end{proof}


\section{Construction of subregular poles}\label{s:Main}
In this section we show that the subregular $L$-factors occuring in table~\ref{tab:subregular_poles} actually exist by verifying the sufficient condition of lemma~\ref{lem:suff_cond_subregular_poles}.

\subsection{Non-generic cases}
Fix a non-generic irreducible $\Pi\in\CCC_G(\omega)$, normalized as in \cite[table~1]{RW}, and a smooth character $\rho$ of $k^\times$ such that $\Lambda=\rho\boxtimes\rho^\divideontimes$ defines a split Bessel model of $\Pi$, i.e.\ $\rho\in\Delta_+(\Pi)$. We assume that there is a $\Gl(2)$-equivariant functional $\overline{\Pi}\to (\C,\widetilde{\rho}\circ\det)$ for $\widetilde{\rho}=\nu\rho$, so by corollary~\ref{B} implies that $\Pi$
is of type \nosf{IIIb}, \nosf{IVc}, \nosf{VIc}, \nosf{VId}.
\begin{lem}\label{dime}
The Bessel module $\widetilde{\Pi}=k_\rho(\overline{\Pi})$ is perfect of degree one.
Especially, $\widetilde{\Pi}^S=0$
and $\dim k_{\chi}(\widetilde{\Pi})= 1$ for every smooth character $\chi$ of $\Gl(1)$.
\end{lem}
\begin{proof} 
See corollary 6.10 in \cite{RW}.
\end{proof}

\begin{lem}Under the above assumptions,
the $\Gl(2)$-module $\overline{\Pi}\in\CCC_{\Gl(2)}$ has a quotient $Q$ isomorphic to $(\widetilde{\rho}\circ\det) \otimes M_{(\St:1)}$, for the induced representation
$$M_{(\St:1)} = (\nu^{1/2}\times\nu^{-1/2})\ .$$
\end{lem}

\begin{proof}
Indeed, the central specialization $\zeta_\mu(\overline{\Pi})$ for $\mu=\widetilde{\rho}^2$ has such a quotient by proposition~\ref{prop:central_specialization_non_generic}.
\end{proof}

The Waldspurger-Tunnell functor ${\cal C}_{\Gl(2)} \to {\cal C}_{T}$ for $T=\{\diag(\ast,1)\in\Gl(2)\}$
sends a smooth $\Gl(2)$-module $X$ to its maximal quotient $X_{\widetilde{T},\rho}$ 
on which $\widetilde{T}=\{(\diag(1,\ast))\in \Gl(2)\}$ acts by the character $\rho$.

\begin{lem}\label{lem:hexagon}
Under the above assumptions,
there is a commutative diagram with the canonical projections
\begin{equation*}
\xymatrix{
& {\overline \Pi} \ar[dl] \ar[dr] \in {\cal C}_{\Gl_a(2)}  & \\
{\CCC}_{TS} \ni k_\rho(\overline{\Pi}) \ar[d] & & \quad \ Q\in {\cal C}_{\Gl(2)} \ar[d] \\
{\CCC}_{T} \ni k_{\nu^2\rho}k_\rho(\overline{\Pi}) \ar[dr]_\cong & & \quad  Q_{\widetilde{T},\rho} \in {\cal C}_{T} \ar[dl]^{\cong} \\
          & k_{\nu^2\rho}k_\rho(Q) \ . & } %
\end{equation*}
Both lower diagonal arrows are isomorphisms.
\end{lem}

\begin{proof}
The center acts on $Q$ by the central character $\mu=\nu^2\rho^2=\tilde\rho^2$, so
the diagonal torus $T\widetilde{T}$ of $\Gl(2)$ acts on 
$Q_{\widetilde{T},\rho}$ by the character $\nu^2\rho\boxtimes\rho$.
It is well-known that $Q_{\widetilde{T},\rho}$ is one-dimensional, compare \cite[lemma~5.3]{W_Excep}.
By construction, $k_{\nu^2\rho}k_\rho(Q)$ is one-dimensional, so the natural projection
from $Q_{\widetilde{T},\rho}$ to $k_{\nu^2\rho}k_\rho(Q)$
is an isomorphism.
The Bessel module $\widetilde{\Pi}$ is perfect by lemma~\ref{dime} and this implies $\dim k_{\nu^2\rho}k_\rho(\overline{\Pi})=1$, see \cite[lemma~3.13]{RW}.
Therefore the natural projection from $k_{\nu^2\rho}k_\rho(\overline{\Pi})$ to $k_{\nu^2\rho}k_\rho(Q)$ is also an isomorphism.
By transitivity of coinvariant functors, both sides of the diagram yield the natural projection to the coinvariant quotient $k_{\nu^2\rho}k_{\rho}(Q)$ on which the standard Borel group $B_{\Gl(2)}$ acts with the character $\nu^2\rho\boxtimes\rho$.
\end{proof}

\begin{prop}\label{reg-Int}
Fix a non-generic irreducible representation $\Pi\in\CCC_G(\omega)$ and an unramified character $\rho$ such that
the Bessel module $\widetilde{\Pi}=k_\rho(\overline{\Pi})$ has degree one. In other words, assume that
$\Lambda=\rho\boxtimes\rho^\divideontimes$ defines a split Bessel model for $\Pi$.
If $\Pi$ admits a non-trivial $(H_+,\widetilde{\rho})$-functional, then for $K=\Gl(2,\mathfrak{o}_k)$ the canonical morphism
$$\overline{\Pi}^K \longrightarrow \overline\Pi \longrightarrow\widetilde{\Pi}\longrightarrow  k_{\nu^2\rho}(\widetilde\Pi/\widetilde{\Pi}^S)$$
is nontrivial.
\end{prop}
\begin{proof}
The Bessel module is perfect of degree one by lemma~\ref{dime}, so $\widetilde{\Pi}^S=0$.
Exactness of $K$-invariants implies that the natural projection $\overline{\Pi}^K\to Q^{K}$ is surjective.
The projection $Q\to Q_{\widetilde{T},\rho}$
is nonzero on $Q^K$ as shown in lemma 5.3 of \cite{W_Excep}.
Since the diagram in lemma~\ref{lem:hexagon} is commutative, by going over the left side
of the diagram we obtain the assertion.
\end{proof}

\begin{thm}\label{thm:main_nongeneric}
For all non-generic irreducible $\Pi\in\CCC_G(\omega)$ with a split Bessel model attached to $\Lambda=\rho\boxtimes\rho^\divideontimes$,
the subregular factor $L_\sreg^\PS(s,\Pi,1,\Lambda)$ is given in table \ref{tab:subregular_poles}.
\end{thm}
\begin{proof}
By proposition~\ref{prop:poles_give_functionals}, proposition~\ref{reg-Int} and lemma~\ref{lem:suff_cond_subregular_poles}, there is a one-to-one correspondence between $(H_+,\widetilde{\rho})$-functionals and subregular poles.
The functionals have been classified in corollary~\ref{B}.
\end{proof}

\subsection{Preparations for the generic cases}
In the next section we will verify the criterion of lemma~\ref{lem:suff_cond_subregular_poles} using the Gelfand-Kazdhan filtration in the sense of lemma~\ref{lem:AB-sequence}.
We now discuss its constituents.
Recall that $B_{\Gl_a(2)}S$ denotes $\{\left[\begin{smallmatrix}\ast&\ast\\0&\ast\end{smallmatrix}\middle|\begin{smallmatrix}\ast\\0\end{smallmatrix}\right]\in \Gl_a(2)\}$.

\begin{lem}\label{lem:induced_J}
For every smooth character $\rho$ of $\Gl(1)$ there is an isomorphism of $\Gl_a(2)$-modules $I\cong j_!(\mathbb{E}[\nu^{2}\rho])$ where $I=\ind_{B_{\Gl(2)}S}^{\Gl_a(2)}(\sigma)$ is compactly induced from $\sigma=\mathbb{S}_1\boxtimes\rho\in\CCC_{TS\times \widetilde{T}}$.
\end{lem}
\begin{proof}
By Bruhat decomposition, the double coset space
\begin{equation*}
B_{\Gl(2)}S\setminus \Gl_a(2)/\left\{\left[\begin{smallmatrix}\ast&\ast\\0&1\end{smallmatrix}\middle|\begin{smallmatrix}\ast\\\ast\end{smallmatrix}\right]\right\}
\end{equation*}
is represented by $\id$ and $\left[\begin{smallmatrix}0&1\\1&0\end{smallmatrix}\middle|\begin{smallmatrix}0\\0\end{smallmatrix}\right]$. The orbit filtration of Bernstein and Zelevinski \cite[5.2]{Bernstein-Zelevinsky77} yields an exact sequence of $TU$-modules
\begin{equation*}                                                                         
0\to \ind_{T}^{TU}(\mathbb{S}_{S,\psi}\otimes \nu\rho) \to j_!(I)\to \ind_{TU}^{TU}(\mathbb{S}_S\otimes i_*(\nu\rho))\to 0\ .
\end{equation*}
Since $\mathbb{S}_S=0$ and $\mathbb{S}_{S,\psi}\cong\C$ we obtain an isomorphism
$$j_!(I)\cong \ind_{T}^{TU}(\nu\rho)\ .$$
Finally, $\ind_{T}^{TU}(\nu\rho)$ is isomorphic to the space 
$\nu^2\rho\otimes\mathbb{E}\cong \mathbb{E}[\nu^2\rho]$ in the sense of \cite[example 3.4]{RW}.
An analogous argument shows $i^*(I)=0$. %
The functorial short exact sequence of lemma~\ref{lem:Gelfand_Kazdhan} implies $I\cong j_!(\mathbb{E}[\nu^2\rho])$.
\end{proof}

\begin{prop}\label{prop:key_result}
Fix unramified characters $\rho$ and $\chi$ of $\Gl(1)$, the representation
$I=j_!(\mathbb{E}[\nu^2\rho])$ of $\Gl_a(2)$, let $\widetilde{I}=k_\rho(I)$ and fix the maximal compact subgroup $K=\Gl(2,\mathfrak{o}_k)$ of $\Gl(2)$.
Then the composition of the natural morphisms
\begin{equation*}
I^K\longrightarrow I\longrightarrow \widetilde{I}\longrightarrow (\widetilde{I}/\widetilde{I}^S)_\chi
\end{equation*}
is non-zero.
\end{prop}
\begin{proof}
We can assume $\rho=1$ by a suitable twist.
By lemma~\ref{lem:induced_J} it is sufficient to show the assertion for the $\Gl_a(2)$-module $I=\ind_{B_{\Gl(2)}S}^{\Gl_a(2)}({\sigma})$ for $\sigma=\mathbb{S}_1\boxtimes\rho\in \CCC_{TS\times \widetilde{T}}$.
Evaluation at identity defines a $B_{\Gl(2)}S$-equivariant projection \begin{equation*}\eval_\id:I\longrightarrow {\sigma}\quad,\quad f\longmapsto f(\id)\ .\end{equation*}
The action of $\widetilde{T}U= \left\{\left[\begin{smallmatrix}1&\ast\\0&\ast\end{smallmatrix}\right]\right\}\subseteq \Gl(2)$ on $\sigma$ is trivial,
so $\eval_\id$ factorizes over the projection to the Bessel module $\widetilde{I}$.
By \cite[theorem~4.27]{RW} there is an isomorphism $\widetilde{I} \cong \mathbb{S}_1\oplus i_*(\nu^2)$. The $TS$-module $\sigma$ is isomorphic to $\mathbb{S}_1$, so the kernel of $\widetilde{I}\twoheadrightarrow\sigma$ equals $\widetilde{I}^S\cong i_*(\nu^2)$.
This defines a commutative diagram of $TS$-modules with an exact horizontal sequence
\begin{equation*}
 \xymatrix{
 & & I\ar[d]\ar[dr]^{\eval_\id} & &\\
0\ar[r]&\widetilde{I}^S\ar[r]&\widetilde{I}\ar[r]&\sigma\ar[r]&0
 }
\end{equation*}
It remains to be shown that $I^K\longrightarrow I\stackrel{\eval_{\id}}{\longrightarrow} \sigma\longrightarrow\sigma_\chi$
is non-trivial.
Indeed, for the characteristic function $\phi_0\in C_c^\infty(k^\times)\cong\mathbb{S}_1$ of $\mathfrak{o}^\times$, define $f^\circ\in I$ by
\begin{equation*}
 f^\circ(g)=\begin{cases}
 {\sigma}(bs)\phi_0 & g=bsk\in B_{\Gl(2)}SK\ ,
 \\0& \text{ otherwise}\ .
\end{cases}
\end{equation*}
This is well-defined because $\phi_0$ is invariant under $B_{\Gl(2)}S\cap K$ 
by construction, so $f^\circ$ is a $K$-invariant element in $I^K$.
The evaluation $\eval_\id$ sends $f^\circ$ to $f^\circ(\id)=\phi_0$.
It is clear that $\phi_0$ has non-trivial image under the projection \begin{equation*}
\mathbb{S}_1\to (\mathbb{S}_1)_\chi\quad,\quad\phi\mapsto \int_{k^\times}\phi(x) \chi^{-1}(x)\Diff^\times x \ .\qedhere
\end{equation*}
\end{proof}

\begin{lem}\label{lem:comp_J_version2}
Fix an unramified character $\rho$ of $\Gl(1)$ and a finite-dimensional $\Gl(1)$-module $A$ that does not contain $\nu\rho$ as a constituent.
For $I=j_!i_*(A)\in\CCC_{2}$ and $K=\Gl(2,\mathfrak{o}_k)$, the composition of canonical morphisms
$$I^K\longrightarrow I \longrightarrow k_\rho(I) \longrightarrow k_{\nu^2\rho} k_\rho (I)$$
is surjective.
It is non-zero if and only if $\nu^2\rho$ occurs in $A$ as a constituent.
\end{lem}
\begin{proof}
By compactness of $K$ and transitivity of coinvariant functors, it is sufficient to prove that the composition $\widehat{I}^K\to\widehat{I}\to k_\rho(\widehat{I})$ is surjective for the central specialization $\widehat{I}=\zeta_{\rho^2\nu^2}(I)$.
We use induction over the length of $A$.
For $A=0$ the assertion is obvious. If $A=\mu$ is a character, then lemma~\ref{lem:cent_loc_j_!i_*} implies
$$\widehat{I}\cong\ind_{B_{\Gl(2)}}^{\Gl(2)}(\mu\boxtimes\nu^2\rho^2\mu^{-1})\ .$$
The well-known formulas for the Jacquet-module show that $k_\rho(\widehat{I})$ vanishes for $\mu\neq\nu\rho,\nu^2\rho$.
If $\mu=\nu^2\rho$, then $k_\rho(\widehat{I})\cong\widehat{I}_{\widetilde{T},\rho}$ and lemma~5.3 in \cite{W_Excep} implies
that $\widehat{I}^K\to \widehat{I}\to k_\rho(\widehat{I})$ is surjective.

If $A$ has length greater than one,
then there is an exact sequence of non-zero $\Gl(1)$-modules
$0\to A_1\to A\to A_2\to0$.
Lemma~\ref{list}.4 yields an exact sequence
$0\to \widehat{I_1}\to \widehat{I}\to \widehat{I_2}\to0$ with $\widehat{I_i}=\zeta_{\nu^2\rho^2}j_!i_*(A_i)$.
By compactness of $K$ and right-exactness of $k_\rho$ we obtain a commutative diagram with exact sequences
\begin{equation*}
\xymatrix
{
0\ar[r] & \widehat{I_1}^K\ar[r]\ar[d] & \widehat{I_{\phantom{1}}}^K\ar[r]\ar[d] & \widehat{I_2}^K\ar[r]\ar[d] & 0\\
 \cdots\ar[r] & k_\rho(\widehat{I_1})\ar[r] & k_\rho(\widehat{I})\ar[r] & k_\rho(\widehat{I_2})\ar[r] & 0\ .\\
}
\end{equation*}
The induction hypothesis states that the left and right vertical arrows are surjective.
By the five lemma the middle vertical arrow is surjective.
\end{proof}

\begin{lem}\label{lem:exceptional_B}
Fix a generic irreducible $\Pi\in\CCC_{G}(\omega)$ and a smooth character $\rho$ of $\Gl(1)$ not in $\Delta_+(\Pi)$.
For $B=i^*(\overline{\Pi})\in\CCC_{\Gl(2)}$ there are $\CCC_1$-isomorphisms
\begin{equation*}
 k^\rho(i_*(B))\cong k_\rho(i_*(B))\cong
\begin{cases}
i_*(\nu^2\rho)& \text{type \nosf{IIIa, IVa, VIa}}\ \text{and}\ \rho\in\Delta_-(\Pi),\\
0         & \text{otherwise}\ .
\end{cases}
\end{equation*}
\end{lem}

\begin{proof}
$k^\rho i_*(B)\cong k_\rho i_*(B)$ holds by lemma~4.3 in \cite{RW}.
If $\Pi$ is cuspidal or of type \nosf{X} or \nosf{XIa}, there is nothing to show because $B=0$.
For type \nosf{VII, VIIIa, IXa} the assertion is also clear because $B$ is cuspidal.
Suppose $\Pi$ is of type \nosf{I, IIa} or \nosf{Va}. Every irreducible constituent of the normalized Siegel-Jacquet-module $\delta_P^{-1/2}\otimes J_P(\Pi)$ is generic and non-cuspidal \cite[table A.3]{Roberts-Schmidt}.
Since $T=\{\diag(\ast,\ast,1,1)\in G\}$ is in the center of the Levi-quotient of $P$, the following characters coincide:
\begin{enumerate}
\item characters in $\Delta_0(\Pi)$,
\item $T$-eigencharacters of the twisted Jacquet-module of $\delta_P^{-1/2}\otimes J_P(\Pi)_\psi$,
\item $T$-eigencharacters of the Siegel-Jacquet-module $\delta_P^{-1/2}\otimes J_P(\Pi)$,
\item $T$-eigencharacters of the Borel-Jacquet-module $\delta_{B_G}^{-1/2}\otimes J_{B_G}(\Pi)$.
\end{enumerate}

In other words, $\Delta_0(\Pi)$ is the multiset of $\{\diag(\ast,1)\subseteq \Gl(2)\}$-eigencharacters on the normalized Borel-Jacquet-module of the $\Gl(2)$-module $B=J_Q(\Pi)$.
Every constituent of $B$ is fully induced as a $\Gl(2)$-module \cite[table A.4]{Roberts-Schmidt},
so the $T$-eigencharacters on $\delta_{B_G}^{-1/2}\otimes B_U$ coincide with the
$\widetilde{T}$-eigencharacters for the torus $\widetilde{T}=\{\diag(1,\ast)\subseteq \Gl(2)\}$.
On $\widetilde{T}$ the modulus character $\delta_{B_G}$ coincides with the valuation character $\nu$.
That means $k_\rho(B)\neq0$ implies that $\rho$ is a $\widetilde{T}$-eigencharacter of the unnormalized Borel-Jacquet-module of $B$, so $\nu^{-1/2}\rho\in\Delta_0(\Pi)$.
Hence $k_\rho i_*(B)\neq0$ is only possible for $\rho\in \Delta_+(\Pi)=\nu^{1/2}\Delta_0(\Pi)$.

If $\Pi$ is of type \nosf{VIa}, then $B$ is isomorphic to $B\cong (\nu^{3/2}\sigma\circ\det)\otimes(1\times1)\oplus \St(\nu\sigma)$ \cite[table~A.4]{Roberts-Schmidt}.
Hence $k_\rho i_*(B)\neq0$ is only possible for $\rho=\nu\sigma\in\Delta_+(\Pi)$ and $\rho=\sigma\in\Delta_-(\Pi)$. It is straightforward to see $k_{\rho}i_*(B)=i_*(\nu^2\sigma)$ for $\rho=\sigma$. 
For type \nosf{IIIa, IVa} the argument is basically analogous, but for type \nosf{IIIb} the constituent $\nu^2\rho\boxtimes\rho$ may occur more than once in the Jacquet module of $B$. It sufficies to show that $\dim k^\rho(B)\leq1$.
Indeed, for type \nosf{IIIa} there is an exact sequence
$$0= k^\rho(\overline{\Pi})\to k^\rho(B)\to k_\rho(J)\to k_\rho(\overline{\Pi})\to 
k_\rho(\overline\Pi)\to k_\rho(B)\to 0$$ where $J=j_!j^!(\overline{\Pi})\cong j_!(\mathbb{E}[A])$ by lemma~4.14 in \cite{RW}.
The constituents of $A=i^*j^!(\overline{\Pi})$ are pairwise distinct, so lemma~4.27 in \cite{RW} ensures that the maximal finite-dimensional subspace of $k_\rho(J)$ is one-dimensional.
\end{proof}

\subsection{Generic cases}
Let $(\Pi,\rho)$ be a non-ordinary or extra-ordinary exceptional case in the sense of \cite[cor.~A.12]{RW}. In other words, $\Pi$ is of type \nosf{I, IIa, Va, VIa, X, XIa} and $\nu^{1/2}\rho\in\Delta_0(\Pi)$,
where $\Delta_0(\Pi)$ denotes the multiset of characters that occur as constituents of $i^*j^!(\overline{\Pi})\cong J_P(\Pi)_\psi$, see \cite[table~3]{RW}.
We show that the subregular factor attached to $\Lambda=\rho\boxtimes\rho^\divideontimes$ is $L_{\sreg}^{\PS}(s,\Pi,1,\Lambda)=L(s,\chi_{\crit})$ for the critical character $\chi_{\crit}=\nu^{1/2}\rho$ by verifying the sufficient criterion of lemma~\ref{lem:suff_cond_subregular_poles}.
For technical reasons we distinguish the cases where the critical character occurs in $\Delta_0(\Pi)$ once and more than once.
Recall that $\Delta_\pm(\Pi)=\nu^{\pm1/2}\Delta_0(\Pi)$.

\textbf{Case 1: The critical character occurs in $\Delta_0(\Pi)$ exactly once.}

\begin{prop}\label{prop:non_degen_generic_except}
For every generic exceptional case $(\Pi,\rho)$ where the critical character 
$\chi_{\crit}=\nu^{1/2}\rho$ occurs in $\Delta_0(\Pi)$ exactly once,
the composition of the canonical morphisms for $K=\Gl(2,\mathfrak{o}_k)$ and $\widetilde{\Pi}=k_\rho(\Pi)$
\begin{equation*}
 \overline{\Pi}^K\longrightarrow \overline{\Pi}\longrightarrow \widetilde{\Pi}  \longrightarrow k_{\nu^2\rho}(\widetilde{\Pi}/\widetilde{\Pi}^S)
\end{equation*}
is non-trivial.
\end{prop}
\begin{proof}By assumption $\Pi$ is of type \nosf{I, IIa, Va, X} or \nosf{XIa}.
Further, $A=i^*j^!(\overline{\Pi})$ splits as a direct sum $A\cong \nu^2\rho\oplus A'$ for some $\Gl(1)$-module $A'$ that does not contain $\nu^2\rho$ as a constituent. By \cite[4.14]{RW}, $j^!(\overline{\Pi})$ is perfect, so there is an embedding $\mathbb{E}[\nu^2\rho]\hookrightarrow j^!(\overline{\Pi})$. Set $J=j_!(\mathbb{E}[\nu^2\rho])$, then exactness of $j_!$ and lemma~\ref{lem:AB-sequence} yield an exact sequence
\begin{equation*}
 0\longrightarrow J\longrightarrow \overline{\Pi}\longrightarrow Y\longrightarrow 0
\end{equation*}
where the cokernel $Y$ is an extension $0\to j_!i_*(A')\to Y\to i_*(B)\to 0$.
By lemma~\ref{lem:exceptional_B},
$k_\rho(B)=0$ and $k^\rho(B)=0$ both vanish, so
$k^\rho(Y)\cong k^\rho j_!i_*(A')$ and $k_\rho (Y)\cong k_\rho j_!i_*(A')$.
The intersection $\Delta_+(\Pi)\cap\Delta_-(\Pi)$ is empty by lemma~A.8 in \cite{RW}, so $\nu\rho$ is not a constituent of $A'$. Lemma~4.28 in \cite{RW} implies
\begin{equation*}k^\rho(Y)=0\qquad,\qquad k_\rho(Y)\cong i_*(A')\ .
\end{equation*}
Consider the following commutative diagram with the canonical arrows

\begin{equation*}
 \xymatrix{
 J^K\ar[r]\ar[d] &J\ar[r]\ar[d] & \widetilde{J}\ar[r]\ar[d] & \widetilde{J}/\widetilde{J}^S\ar[r]\ar[d] & k_{\nu^2\rho}(\widetilde{J}/\widetilde{J}^S)\ar[d]^\cong \\
 \overline{\Pi}^K\ar[r] &\overline{\Pi}\ar[r] & \widetilde{\Pi}\ar[r] & \widetilde{\Pi}/\widetilde{\Pi}^S\ar[r] & k_{\nu^2\rho}(\widetilde{\Pi}/\widetilde{\Pi}^S)\ .
 }
\end{equation*}

We have shown above that the third vertical arrow is injective with cokernel $k_\rho(Y)\cong i_*(A')$.
The fourth vertical arrow is injective with cokernel $i_*(A')$ because $\widetilde{J}^S\cong\widetilde{\Pi}^S$ is the unique submodule isomorphic to $i_*(\nu^2\rho)$ as shown in \cite[thm.~4.27, thm.~6.6]{RW}.
Since $\nu^2\rho$ does not occur in $A'$, both $$k_{\nu^2\rho}(i_*(A'))=0\quad\text{and}\quad k^{\nu^2\rho}(i_*(A'))=0$$ vanish, so the fifth vertical arrow is an isomorphism.
By proposition~\ref{prop:key_result} the composition of the upper horizontal arrows is non-trivial. By commutativity of the diagram the composition of the lower horizontal arrows is non-trivial.
\end{proof}

\textbf{Case 2: The critical character occurs in $\Delta_0(\Pi)$ more than once.}

Now we consider the remaining case and thus make the following

\textbf{Assumption:} Fix a generic irreducible representation $\Pi\in\CCC_G(\omega)$ and a smooth character $\rho$ of $\Gl(1)$ such that
\begin{enumerate}\item $\Hom_{\Gl(2)}(\overline{\Pi},\widetilde{\rho}\circ\det)$ is non-zero for $\widetilde\rho=\nu\rho$,
\item $\chi_\crit=\nu^{1/2}\rho$ occurs in the multiset $\Delta_0$ with multiplicity $\geq2$.
\end{enumerate}
Under these assumptions, the Bessel module $\widetilde{\Pi}=k_\rho(\overline{\Pi})\in\CCC_1$ is not perfect \cite[6.6]{RW}. It contains a unique one-dimensional submodule $\widetilde{\Pi}^S\cong i_*(\nu^2\rho)$.
The main idea now is to consider the commutative diagram of proposition~\ref{prop:VIa_surjection_version2}. Its construction is explained in the following lemmas.
For $B=i^*(\overline{\Pi})$ and $J=j_!j^!(\overline{\Pi})$, lemma~\ref{lem:Gelfand_Kazdhan} gives an exact sequence in $\CCC_2$
 \begin{equation*}
  0\longrightarrow J \stackrel{a}{\longrightarrow} \overline{\Pi} \longrightarrow i_*(B) \longrightarrow 0\ .
 \end{equation*}

\begin{lem}\label{lem:VIa_second_filtration}
 $J\cong j_!(\mathbb{E}[A])$ is perfect.
\end{lem}
\begin{proof}
Since $\overline{\Pi}$ is generic, $j^!(\overline{\Pi})$ is perfect of degree one \cite[4.13]{RW}.
It is uniquely determined by $A=i^*j^!(\overline{\Pi})$ and thus isomorphic to $\mathbb{E}[A]$.
The functor $j_!$ sends perfect modules to perfect ones \cite[lemma~4.4]{RW}.
\end{proof}

\begin{lem}\label{lem:diagonal_arrow_VIa_version2}
For the Bessel module $\widetilde{J}=k_\rho(J)$,
the composition $g=p\circ k_\rho(a)$ of the canonical $\CCC_1$-morphisms
\begin{equation*}
\xymatrix@+5mm
{
g: \ \widetilde{J}\ar[r]^{k_\rho(a)} & \widetilde{\Pi}\ar[r]^p & \widetilde{\Pi}/\widetilde{\Pi}^S\ 
}
\end{equation*}
is surjective .
\end{lem}

\begin{proof}
The functor $k_\rho$ applied to the exact sequence of lemma~\ref{lem:VIa_second_filtration} yields a long exact sequence
\begin{equation*}
\dots\longrightarrow k^\rho(\overline{\Pi}) \longrightarrow k^\rho i_*(B) \stackrel{\delta}{\longrightarrow} k_\rho(J)
\stackrel{k_\rho(a)}{\longrightarrow} k_\rho(\overline{\Pi}) \longrightarrow k_\rho i_*(B)\longrightarrow 0\ .
\end{equation*}
The left term $k^\rho(\overline{\Pi})$ vanishes because $\Pi$ is generic \cite[lemma~5.10]{RW}.
Lemma~\ref{lem:exceptional_B} implies
$$k^\rho(i_*(B))\cong k_\rho(i_*(B))\cong
\begin{cases}
i_*(\nu^2)& \text{type \nosf{VIa}}\ ,\\
0         & \text{otherwise}\ .
\end{cases}$$
For every case not of type \nosf{VIa} this shows that $k_\rho(J)\to k_\rho(\overline{\Pi})$ is an isomorphism,
which yields the assertion.
For type \nosf{VIa}, lemma~4.29 in \cite{RW} implies that $k_\rho(J)$ has a unique one-dimensional submodule isomorphic to $i_*(\nu^2)$ and fits in an exact sequence in $\CCC_1$
\begin{equation*}
 0\to i_*({\nu^2}) \to k_{\rho}(J)\to \mathbb{E}[\nu^2] \to 0\ .
\end{equation*}
Especially, the cokernel of $\delta$, or equivalently the image of $k_\rho(a)$, is isomorphic to $\mathbb{E}[\nu^2]$.
Since $\widetilde{\Pi}\cong i_*(\nu^2)\oplus \mathbb{E}[\nu^2]$ by \cite[theorem~6.6]{RW} and since $i_*(\nu^2)$ is the kernel of $p$, the assertion follows.
\end{proof}

\begin{lem}\label{lem:factor_over_breve_version2}
There is a unique $\CCC_1$-morphism $h$ that makes the following
diagram with the canonical surjective arrows in $\CCC_1$ commute,
\begin{equation*}
 \xymatrix@+3mm
 {
 \widetilde{J}\ar@{->>}[r]\ar@{->>}[d]_g  & (\widetilde{J/\mathbb{S}_2}) \ar@{.>>}[d]^{\exists! h} \\
 \widetilde{\Pi}/\widetilde{\Pi}^S\ar@{->>}[r]&  \pi_0(\widetilde{\Pi}/\widetilde{\Pi}^S)\ .
 }
 \end{equation*}
Especially, $h$ is surjective.
\end{lem}

\begin{proof}It suffices to show the existence of $h$, then uniqueness and surjectivity follow from surjectivity of $g$ shown in lemma~\ref{lem:diagonal_arrow_VIa_version2}.
Consider the following commutative diagram in $\CCC_1$
\begin{equation*}
\xymatrix{
 \widetilde{J/\mathbb{S}_2}\ar[d]\ar@{.>}[drr]_(0.3)h & \widetilde{J}\ar[r]^g\ar[d]\ar[l] & \widetilde{\Pi}/\widetilde{\Pi}^S\ar[d] \\
\pi_0(\widetilde{J/\mathbb{S}_2}) & \pi_0(\widetilde{J})\ar[l]_\cong\ar[r] & \pi_0(\widetilde{\Pi}/\widetilde{\Pi}^S)\ .
}
\end{equation*}
where the tilde denotes the functor $k_\rho$.
Surjectivity of $g$ has been shown in lemma~\ref{lem:diagonal_arrow_VIa_version2} and surjectivity of the other arrows follows from right-exactness of $k_\rho$.
It only remains to be shown that the lower left horizontal arrow is an isomorphism.
Indeed, $k_\rho$ applied to the exact sequence
$0 \to \mathbb{S}_2 \to J \to J/\mathbb{S}_2 \to 0$
yields by \cite[4.7]{RW} a long exact sequence
\begin{equation*}
\dots \longrightarrow  \mathbb{S}_1 \longrightarrow \widetilde{J} \longrightarrow \widetilde{J/\mathbb{S}_2} \longrightarrow 0\ .
\end{equation*}
The exact functor $\pi_0=i_*i^*$ applied to this sequence shows that the natural morphism
$$\pi_0 (\widetilde{J})\to \pi_0 (\widetilde{J/\mathbb{S}_2})$$ is an isomorphism because $\pi_0(\mathbb{S}_1)$ vanishes.
Thus the diagonal arrow $h$ exists.
\end{proof}

\begin{prop}\label{prop:VIa_surjection_version2}
Fix a generic $\Pi\in\CCC_G(\omega)$ and a smooth character $\rho$ of $\Gl(1)$ as in the above assumption.
Then the composition of the canonical morphisms
\begin{equation*}
\overline{\Pi}^K\longrightarrow \overline{\Pi}\longrightarrow\widetilde{\Pi} \longrightarrow k_{\nu^2\rho}(\widetilde{\Pi}/\widetilde{\Pi}^S)
\end{equation*}
is non-zero for $K=\Gl(2,\mathfrak{o})$ and the Bessel module $\widetilde{\Pi}=k_\rho(\overline{\Pi})$.
\end{prop}
\begin{proof}
Let $J=j_!j^!(\overline{\Pi})$ and $\chi=\nu^2\rho$.
The functor of $K$-invariants and the coinvariant functors applied to the embedding $a:J\hookrightarrow\overline{\Pi}$ yield a commutative diagram of complex vector spaces
with the canonical arrows
\begin{equation*}
\xymatrix{
  J^K\ar[r]\ar@{^{(}->}[d]      &   J\ar[r]\ar@{^{(}->}[d]^a  & \widetilde{J}\ar[r]\ar@{->>}[d]^g& \widetilde{J/\mathbb{S}_2}\ar[r]\ar@{->>}[d]^{h} & k_\chi(\widetilde{J/\mathbb{S}_2})\ar@{->>}[d]^{k_\chi(h)}\\
  \overline{\Pi}^K \ar[r]& \overline{\Pi}\ar[r] & \widetilde{\Pi}/\widetilde{\Pi}^S \ar[r]\ar[dr]  &  \pi_0(\widetilde{\Pi}/\widetilde{\Pi}^S)\ar[r] & k_\chi\pi_0(\widetilde{\Pi}/\widetilde{\Pi}^S)\\
               &   &   &     k_\chi(\widetilde{\Pi}/\widetilde{\Pi}^S)\ar[ur] \ .
}
\end{equation*}
The third vertical arrow $g$ is surjective by lemma~\ref{lem:diagonal_arrow_VIa_version2}.
The fourth vertical arrow $h$ exists and is surjective by lemma~\ref{lem:factor_over_breve_version2}.
By right-exactness of $k_\chi$, the fifth vertical arrow $k_\chi(h)$ is also surjective.
The composition of the upper horizontal arrows is surjective by lemma~\ref{lem:comp_J_version2} and because $J^K\to (J/\mathbb{S}_2)^K$ is surjective.
Hence the composition of the lower horizontal arrows is surjective.
By assumption, $\chi$ occurs at least once in $\pi_0(\widetilde{\Pi}/\widetilde{\Pi}^S)$,
so $k_\chi\pi_0(\widetilde{\Pi}/\widetilde{\Pi}^S)$ is non-zero.
The lower triangle is commutative by transitivity of coinvariant functors.
Going over the diagonal arrows shows the assertion.
\end{proof}

\begin{thm}\label{thm:main_generic}
Fix a generic irreducible $\Pi\in\CCC_G(\omega)$ and a smooth character $\rho$ of $\Gl(1)$.
If $(\Pi,\rho)$ is non-ordinary exceptional or extra-ordinary exceptional in the sense of \cite{RW},
then the subregular $L$-factor with split Bessel model attached to the Bessel character $\Lambda=\rho\boxtimes\rho^\divideontimes$ is
\begin{equation*}L_{\sreg}^{\PS}(s,\Pi,1,\Lambda)=L(s,\chi_{\crit})\quad,\quad \chi_{\crit}=\nu^{1/2}\rho\ .
\end{equation*}
If neither $(\Pi,\rho)$ nor $(\Pi,\rho^\divideontimes)$ are non-ordinary exceptional or extra-ordinary exceptional,
then the subregular $L$-factor $L_\sreg^\PS(s, \Pi,1,\Lambda)=1$ is trivial.
\end{thm}

\begin{proof}
Proposition~\ref{prop:non_degen_generic_except} and proposition~\ref{prop:VIa_surjection_version2} imply that the sufficient condition of lemma~\ref{lem:suff_cond_subregular_poles} is satisfied, so there is a subregular pole of the form $L(s,\nu^{1/2}\rho)$.
Any additional subregular pole would be of the form $L(s,\nu^{1/2}\rho^\divideontimes)$ where $\rho\neq\rho^\divideontimes$ by lemma~\ref{2.1}. This would give rise to an $(H_+,\nu\rho^\divideontimes)$-functional by proposition~\ref{prop:poles_give_functionals}.
Then proposition~\ref{EXGEN} implies that the two distinct characters $\rho,\rho^\divideontimes$ are both contained in $\Delta_-(\Pi)$.
This contradicts \cite[lemma~A.9]{RW}.

The last assertion follows from proposition~\ref{prop:poles_give_functionals} and proposition~\ref{EXGEN}.
\end{proof}

Theorem~\ref{thm:main_nongeneric} and theorem~\ref{thm:main_generic} imply
that the subregular poles are exactly the ones listed in table~\ref{tab:subregular_poles}. This completes the results of \cite{RW} and \cite{W_Excep} and shows that
the Piatetski-Shapiro $L$-factors $L^{\PS}(s,\Pi,\mu,\Lambda)$ coincide with the $L$-factors  formally defined in \cite{Roberts-Schmidt}, table A.8 using the Langlands classification (see \cite{Roberts-Schmidt}, section 2.4 of loc.~cit.~as well as Gan-Takeda \cite{Gan_Takeda_LLC}.
This proves our final result:

\begin{thm}\label{thm:main}
For all irreducible smooth representation $\Pi$ of $\GSp(4,k)$ and all
split Bessel models $\Lambda=\rho\boxtimes\rho^\divideontimes$ of $\Pi$
with Bessel module $\widetilde{\Pi}=k_\rho(\overline{\Pi})$,
the Piatetski-Shapiro $L$-factor $$L^{\PS}(s,\Pi,\mu,\Lambda)=
L(s,\mu\nu^{-3/2}\otimes (\widetilde{\Pi}/\widetilde{\Pi}^S))L_{\sreg}^{\PS}(s,\Pi,\mu,\Lambda)L^{\PS}_{\ex}(s,\Pi,\mu,\Lambda)$$
is the expected $L$-factor defined for $(\Pi,\mu)$ given by the Langlands philosophy. It does not depend on the particular choice of the split Bessel model. 
\end{thm}

For the analogous results on anisotropic Bessel models see Dani\c{s}man~\cite{Danisman, Danisman_Annals, Danisman2, Danisman3} and the authors \cite{Anisotropic_Exceptional}. Summarily, it follows that the Piatetski-Shapiro $L$-factor is independent of any Bessel model, the anisotropic or split.

\appendix

\section{Appendix}
\subsection{Central specialization}\label{ss:central_specialization}
Instead of the category $\CCC_{\Gl(n)}$ it is sometimes preferable to consider its full subcategory $\CCC=\CCC_{\Gl(n)}(\mu)$ of smooth representations with central character $\mu$.
For example, $\CCC^\fin_{\Gl(n)}(\mu)$ has cohomological dimension $\leq n$ \cite[II.3.3]{Schneider-Stuhler_Bruhat_Tits} and simplifies the consideration of the Euler characteristic.
For the fixed smooth character $\mu$ of the center $Z\cong k^\times$ of $\Gl(n)$, consider the functor
$$\zeta_\mu: {\cal C}_{\Gl(n)} \to {\cal C}_{\Gl(n)}(\mu)$$
sending $\Gl(n)$-modules $M$ to their $(Z,\mu)$-coinvariant quotient $M_{(Z,\mu)}$.
The functor $\zeta_\mu$ is right-exact and its left-derived functor is the functor of $(Z,\mu)$-invariants
$\zeta^\mu: \CCC_{\Gl(n)}\to \CCC_{\Gl(n)}(\mu)\,,\ M\mapsto M^{(Z,\mu)}$, compare \cite[lemma~A.2]{RW}.
By composition with the forgetful functor
${\CCC}_n \to {\CCC}_{\Gl(n)}$
we obtain functors ${\cal C}_{n} \to {\cal C}_{\Gl(n)}(\mu)$, also denoted $\zeta_\mu, \zeta^\mu$.
We use the analogous notation for representations of subgroups of $\Gl(n)$ containing $Z$.
The meaning will be clear from the context.

\begin{lem}\label{list}
\begin{enumerate}
\item The functor $\zeta_\mu$ is left-adjoint and the functor $\zeta^\mu$ is right-adjoint to the natural embedding $\CCC_{\Gl(n)}(\mu)\to\CCC_{\Gl(n)}$.
\item For every closed subgroup $M\subseteq\Gl(n)$ with $Z\subseteq M$ there are natural equivalences 
\begin{equation*}
\zeta_\mu \circ\ind_M^{\Gl(n)} \cong \ind_M^{\Gl(n)}\circ\zeta_\mu\qquad,\qquad\zeta^\mu \circ\ind_M^{\Gl(n)} \cong \ind_M^{\Gl(n)}\circ\zeta^\mu \ .
\end{equation*}
\item 
The functor $\zeta^\mu:\CCC_n\to\CCC_{\Gl(n)}(\mu)$ factorizes over the functor $\kappa:\CCC_n\to\CCC_{\Gl(n)}$
of invariants under $\ker(\Gl_a(n)\to\Gl(n))$.

\item The functor $\zeta_\mu\circ j_!$ is exact and 
$\zeta^\mu\circ j_!=0$ is zero.
\item There is an isomorphism $\zeta_\mu(X)\cong \zeta^\mu (X) $ for irreducible $X\in\CCC_{\Gl(n)}$.
\end{enumerate}
\end{lem}
\begin{proof}
This is a straightforward consequence of results in \cite[\S X]{Borel_Wallach}.
\end{proof}
\begin{rmk}\label{rmk:cent_loc_fin_length}
The functor $\zeta^\mu: {\cal C}_n^{\fin} \to {\cal C}_{\Gl(n)}^{\fin}$ preserves finite length.
The functor $\zeta_\mu:\CCC^\fin_n\to\CCC_{\Gl(n)}(\mu)$ does not preserve finite length.
Indeed, $\zeta_\mu(\mathbb{S}_2)\in\CCC_{\Gl(2)}(\mu)$ does not have finite length because $k_\rho\zeta_{\mu}(\mathbb{S}_2)$ is non-zero for every $\rho$ by corollary~\ref{cor:krho_of_centS2}.
However, for $X\in\CCC^{\fin}_2$ with $j^!\circ j^!(X)=0$ the central specialization $\zeta_\mu(X)\in{\CCC}_{\Gl(2)}^{\fin}(\rho)$ does have finite length.
\end{rmk}

\begin{lem}\label{lem:cent_loc_j_!i_*}%
For a smooth character $\chi$ of $k^\times$ let $A=\chi^{(n)}\in\CCC_{\Gl(1)}$
be the attached Jordan block of length $n$.
Then the $\Gl(1)$-module $\zeta_\mu(j_!i_*(A))$ has a
composition series of modules of lenght $n$ with graded pieces isomorphic to
$$\zeta_\mu j_!i_*(\chi) \cong \chi\nu^{-1/2} \times \chi^{-1}\nu^{1/2}\mu\ .$$
Furthermore $\zeta^\mu(j_!i_*(A))=0$.
\end{lem}

\begin{proof}
$\zeta^\mu\circ j_!=0$ vanishes 
and the functor $\zeta_\mu\circ j_! i_*$ is exact by lemma~\ref{list}.
Therefore $\zeta_\mu \circ j_!i_*(\chi^{(n)})$ has a composition series of length $n$  with graded pieces $\zeta_\mu j_!i_*(\chi)$.
Lemma~\ref{list} and transitivity of induction shows that there is an isomorphism
\begin{equation*}
\zeta_\mu j_!i_*(\chi)=\zeta_\mu(\ind^{\Gl(2)}_{B_{\Gl(2)}}(\chi\boxtimes\mathcal{S}))\cong \ind^{\Gl(2)}_{B_{\Gl(2)}}(\zeta_\mu(\chi\boxtimes\mathcal{S}))\ .
\end{equation*} Since $\zeta_\mu(\chi\boxtimes\mathcal{S})=\chi\boxtimes\mu\chi^{-1}$ this shows the statement.
\end{proof}

\begin{lem}\label{lem:cent_loc_S2}
For smooth characters $\mu$ of $k^\times$, the central specializations of $\mathbb{S}_n\in\CCC_n$ are $\zeta^\mu(\mathbb{S}_{n})=0$
and 
\begin{equation*}
\zeta_\mu(\mathbb{S}_n)\cong \ind_{Z_{\Gl(n)}\times\Gl_a(n-1)}^{\Gl(n)}
  \left( \mu\boxtimes\mathbb{S}_{n-1}\right)
\end{equation*}
induced from the parabolic subgroup $Z_{\Gl(n)}\times \Gl_a(n-1)\subseteq \Gl(n)$ where the mirabolic subgroup is $\Gl_a(n-1)=\left(\begin{smallmatrix}\ast&\ast\\0&1\end{smallmatrix}\right)\subseteq \Gl(n)$.
\end{lem}
\begin{proof}
By definition $\mathbb{S}_n=j_!(\mathbb{S}_{n-1})$, so the restriction to $\Gl(n)$ is isomorphic to the compactly induced representation
\begin{equation*}
\mathbb{S}_n|_{\Gl(n)}\cong \ind_{\Gl_a(n-1)}^{\Gl(n)} \mathbb{S}_{n-1}\cong \ind^{\Gl(n)}_{Z\times \Gl_a(n-1)} (\mathcal{S}\boxtimes\mathbb{S}_{n-1})
\end{equation*}
where $\mathcal{S}=C_c(k^\times)=\ind_{\{1\}}^Z(1)\in\CCC_{Z}$.
The central specializations of $\mathcal{S}\boxtimes\mathbb{S}_{n-1}$ are $\zeta_\mu(\mathcal{S}\boxtimes\mathbb{S}_{n-1})\cong \mu\boxtimes \mathbb{S}_{n-1}$ and $\zeta^\mu(\mathcal{S}\boxtimes\mathbb{S}_{n-1})=0$.
By lemma~\ref{list}, central specialization commutes with $\ind_{Z\times \Gl_a(n-1)}$ and this implies the statement.
\end{proof}

The functors $k_\rho, k^\rho:\CCC_n\to\CCC_{n-1}$ are defined in \cite{RW}.
By restriction they define functors $\CCC_{\Gl(n)}\to \CCC_{\Gl(n-1)}$, also denoted $k_\rho$, $k^\rho$.
\begin{lem}\label{lem:cent_loc_com_diagram}
There is a commutative diagram
$$ \xymatrix{ {\cal C}_n \ar[d]_{k_\rho}\ar[r]^{\text{forget}} & {\cal C}_{\Gl(n)} \ar[d]^{k_\rho}\ar[r]^{\zeta_\mu} &  {\cal C}_{\Gl(n)}(\mu) \ar[d]^{k_\rho} \cr
 {\cal C}_{n-1} \ar[r]^{\text{forget}} & {\cal C}_{\Gl(n-1)} \ar[r]^{\zeta_{\mu/\rho}} &  {\cal C}_{\Gl(n-1)}(\mu/\rho) \cr} $$
The functors $k_\rho:\CCC_1\to\CCC_0$ and $\zeta_{\rho}:\CCC_1\to\CCC_{k^\times}(\rho)\cong\CCC_0$ are naturally equivalent.
\end{lem}
\begin{proof}
This follows by transitivity of coinvariant functors.
\end{proof}

\begin{cor}\label{cor:krho_of_centS2}
For $n\geq2$ and smooth characters $\rho$ and $\mu$ of $k^\times$,
\begin{equation*}
 k_\rho\zeta_\mu(\mathbb{S}_n) \cong \zeta_{\mu/\rho}(\mathbb{S}_{n-1})\    \qquad,\qquad  k^\rho\zeta_\mu(\mathbb{S}_n)=0\qquad\qquad \text{ in }\ \CCC_{\Gl(n)}(\mu/\rho)\  .
\end{equation*}
\end{cor}
\begin{proof}
Lemma~\ref{lem:cent_loc_com_diagram} and \cite[cor.~4.7]{RW} show $k_\rho\circ\zeta_\mu(\mathbb{S}_n)\cong \zeta_{\mu/\rho}\circ k_\rho(\mathbb{S}_n)\cong \zeta_{\mu/\rho}(\mathbb{S}_{n-1})$ and thus the first assertion.
For the second assertion we only give the proof for $n=2$.
The proof for $n>2$ is similar using \cite[5.2]{Bernstein-Zelevinsky77}.
By Bruhat decomposition there is an exact sequence of $B_{\Gl(2)}=\Gl_a(1)\times Z$-modules
\begin{equation*}
 0\to \ind_{\Gl(1)\times Z}^{B}(\mathcal{S}\boxtimes\mu)^w \to \zeta_\mu(\mathbb{S}_2) \to (\mathbb{S}_1\boxtimes\mu)\to 0\ .
\end{equation*}
with the Weyl group element $w=\left(\begin{smallmatrix}0&1\\1&0\end{smallmatrix}\right)$ and the $\Gl(1)$-module $\mathcal{S}=\mathbb{S}_1|_{\Gl(1)}$.
Let $U$ be the unipotent radical of $B_{\Gl(2)}$.
The $U$-coinvariants of the right hand side are zero because $(\mathbb{S}_1)_U \cong i^*\mathbb{S}_1$ vanishes.
By integration over $U$, the $U$-coinvariant quotient of the left hand side is isomorphic to $\mathcal{S}\boxtimes \mu$ as a representation of $\{\left(\begin{smallmatrix}1&0\\0&\ast\end{smallmatrix}\right)\}\times Z$. Finally, $\mathcal{S}^\rho=0$ vanishes.
\end{proof}

\begin{lem}{\cite[4.3]{RW}}\label{lem:krho_fin_generated}
There is a natural equivalence  of functors
$$\CCC^\fin_{\Gl(n)}\to \CCC^\fin_{n-1}\qquad\text{by} \qquad k^\rho i_*\cong k_\rho i_*\ .$$
\end{lem}

\subsection{Ext-functor}\label{s:Ext}
For every connected reductive group $X$ over $k$,
the categories ${\CCC}_X$ and ${\CCC}_X(\omega)$ have enough projectives and injectives \cite[\S5.9f]{Vigneras_Representations_Modulaire} and we obtain a well-defined bifunctor $\Ext^n_{\CCC_{X}}(-,-)$.

\begin{lem}[Dual Frobenius reciprocity]\label{lem:Dual_Frob_Ext}
For a closed subgroup $Y\subseteq X$, there is a natural isomorphism
 \begin{equation*}
  \Ext^n_{\CCC_{X}}(V,\Ind_Y^X(W)) \cong \Ext^n_{\CCC_Y}(V|_{Y},W)\ 
 \end{equation*}
 for every $V\in\CCC_X$ and $W\in\CCC_Y$ and every $n\geq0$.
\end{lem}
\begin{proof}
See \cite[5.10]{Vigneras_Representations_Modulaire} and \cite[X.1.7]{Borel_Wallach}. %
\end{proof}

\begin{lem}\label{lem:Ext_Gl1}
For every smooth character $\chi$ of $\Gl(1)$ and $X\in\CCC_{\Gl(1)}$, there are natural isomorphisms
$$\Hom_{\Gl(1)}(X,\chi)\cong \Hom_\C(X_\chi,\C)\quad ,\quad \Ext^1_{\Gl(1)}(X,\chi)\cong \Hom_\C(X^\chi,\C)\ .$$
\end{lem}
\begin{proof}The first assertion is clear.
The $\chi$-invariant functor $X\mapsto X^\chi$ is left-derived to the $\chi$-coinvariant functor $X\mapsto X_\chi$, compare \cite[A.2]{RW}.
Dualization is exact and contravariant, so $X\mapsto\Hom(X^\chi,\C)$ is the right-derived functor of $X\mapsto\Hom(X_\chi,\C)$.
\end{proof}
Especially, for each smooth character $\chi$ there is a unique indecomposable $\Gl(1)$-module $\chi^{(n)}$ of length $n$ whose constituents are isomorphic to $\chi$. It is realized as the Jordan block attached to the monodromy operator $\tau_\chi$. These are all the indecomposable $\Gl(1)$-modules of finite length.
\begin{lem}\label{ext-lemma}
If $\omega=\rho\chi$ is a square, then for every $X\in\CCC_{\Gl(2)}(\omega)$ there are natural isomorphisms
$$  \Ext_{ {\cal C}_{\Gl(2)}(\omega)}^1(X, \ind_B^G(\chi\boxtimes \rho))  \cong \Hom_\C(k^\rho(X),\C) \ ,$$
$$  \Hom_{ {\cal C}_{\Gl(2)}(\omega)}  (X, \ind_B^G(\chi\boxtimes \rho))  \cong \Hom_\C(k_\rho(X),\C) \ .$$
\end{lem}
\begin{proof}
Since $\omega$ is a square, we may assume $\omega=1$ by a twist.
Then ${\cal C}_{\Gl(2)}(1)= {\cal C}_{\mathrm{PGl}(2)}$ and we can apply lemma~\ref{lem:Dual_Frob_Ext},
\begin{gather*} \Ext^1_{\mathrm{PGl}(2)}( X , \ind_B^G( \chi\boxtimes \rho)) = \Ext^1_{B/Z}( X,
 \chi\boxtimes \rho)\\
 = \Ext^1_{T\widetilde{T}/Z}(  X_U,  \chi\boxtimes \rho)
=   \Hom_\C(k^{\rho} (X) , \C) \ ,
\end{gather*}
where the last step is lemma~\ref{lem:Ext_Gl1}.
The second assertion is analogous.
\end{proof}

\begin{lem}\label{lem:cent_spez_indecomposable}
For a smooth character $\chi$ of $\Gl(1)$ consider an extension of smooth representations of $G=\Gl(2)$
\begin{equation*}
 0\to j_!i_\ast(\chi)|_{G}\to Q \to (\chi\circ\det)\to 0\ .
\end{equation*}
If $\zeta^{Z,\mu}(Q)=0$ for $\mu=\chi^2$, then $\zeta_{\mu}(Q)$ is the unique indecomposable extension %
$$0\to \chi\otimes\St \to \zeta_\mu(Q) \to (\chi\circ\det)\to 0 \ .$$
Likewise, if an extension
\begin{equation*}
 0\to j_!i_\ast(\nu\chi)|_{G}\to Q \to \chi\otimes \St \to 0\ .
\end{equation*}
satisfies $\zeta^{Z,\mu}(Q)=0$ for $\mu=\chi^2$, then $\zeta_{\mu}(Q)$ is the unique indecomposable extension 
$$0\to (\chi\circ\det) \to \zeta_\mu(Q) \to \chi\otimes\St\to 0 \ .$$
\end{lem}

\begin{proof}
The first assertion has been proved in lemma~6.4 of \cite{W_Excep}. The proof for the second assertion is completely analogous.
\end{proof}

\subsection{Euler characteristic}\label{ss:Euler_char}
Fix a smooth central character $\mu$ and write $\widehat{X}=\zeta_\mu(X)\in\CCC:={\CCC}^{\fin}_{\Gl(2)}(\mu)$ for the central specialization of $X\in\CCC^\fin_2$.
We consider the Euler characteristic
of $X$ and $Y\in \CCC$
to be that of $\widehat{X}$ and $Y$ in $\CCC$, i.e.
\begin{equation*}
\chi(X,Y) = \dim \Hom_{\CCC}(\widehat{X},Y)-\dim \Ext^1_{\CCC}(\widehat{X},Y)\ .
\end{equation*}
The higher $\Ext_\CCC^i$-classes for $i\geq2$ vanish \cite[II.3.3]{Schneider-Stuhler_Bruhat_Tits}. In other words, the cohomological dimension of $\CCC$ is one.

\begin{lem}
The Euler characteristic $\chi(X,Y)$ is well-defined.
\end{lem}
\begin{proof}
It has to be shown that $\Hom_{\CCC}(\zeta_\mu(X),Y)$ and $\Ext^1_{\CCC}(\zeta_\mu(X),Y)$ is finite-dimensional.
We can assume that $X$ and $Y$ are irreducible.
If $j^!\circ j^!(X)=0$, then $\widehat{X}$ has finite length as a $\Gl(2)$-module by remark~\ref{rmk:cent_loc_fin_length} and the assertion follows.
If $j^!\circ j^!(X)\neq0$ then $X\cong\mathbb{S}_2$
and the assertion follows from lemma~\ref{lem:cent_loc_S2}.
\end{proof}

\begin{lem}\label{lem:Euler_char_additive}
If $Y\in\CCC$ is parabolically induced from a character of $B_{\Gl(2)}$, then
for every exact sequence $0\to X_1\to X\to X_2\to 0$ in $\CCC^\fin_2$ we have
\begin{equation*}
  \chi(X,Y) = \chi(X_1,Y)+\chi(X_2,Y)\ .
\end{equation*}
For every $X\in\CCC^\fin_2$ and every short exact sequence 
$0\to Y_1\to Y\to Y_2\to 0$ in $\CCC$,
\begin{equation*}
  \chi(X,Y) = \chi(X,Y_1)+\chi(X,Y_2)\ .
\end{equation*}
\end{lem}
\begin{proof}
Central specialization yields a long exact sequence in $\CCC$
\begin{equation*}
0\to \zeta^\mu(X_1)\to \zeta^\mu(X) \to \zeta^\mu(X_2) \to \zeta_\mu(X_1)\to \zeta_\mu(X) \to \zeta_\mu(X_2) \to 0\ .
\end{equation*}
Since $\zeta^\mu(X)$ has finite length for every $X\in\CCC_2^\fin$
by remark~\ref{rmk:cent_loc_fin_length},
\begin{equation*}
\chi(\zeta^\mu(X),Y)=\sum_{i=0}^1(-1)^i\dim(\Ext^i_{\CCC}(\zeta^\mu(X),Y)=0
\end{equation*}
vanishes by lemma~\ref{lem:krho_fin_generated} and lemma~\ref{ext-lemma}.
The well-known formulas for the Euler characteristic imply the statement.
\end{proof}

\begin{prop}\label{Euler}
For $X\in {\cal C}_{2}^{\fin}$ and $Y=\ind_{B_{\Gl(2)}}^{\Gl(2)}(\mu\rho^{-1}\boxtimes \rho)$,
the Euler characteristic is $$\chi(X,Y)=  \dim (j^!\circ j^!)(X)\ $$
with the functor $(j^!\circ j^!):\CCC_2\to\CCC_0$. Especially,
$$\dim \Ext^1_{\CCC}(\widehat{X}, Y) = \dim k_\chi k_\rho (X)  - \dim (j^!\circ j^!)(X) \ .$$
\end{prop}

\begin{proof}
By lemma~\ref{lem:Euler_char_additive} it suffices to show the assertion for irreducible $X$.
Note that $\dim\chi(X,Y)=\dim \Hom(k_\rho\zeta_\mu X,\C)-\dim\Hom(k^\rho \zeta_\mu X,\C)$ as in lemma~\ref{ext-lemma}.
For $X=i_*(\pi)$ with irreducible $\pi\in\CCC_{\Gl(2)}$ use lemma~\ref{lem:krho_fin_generated}. The argument for $X=j_!i_*(\chi)$ is similar using lemma~\ref{lem:cent_loc_j_!i_*}.
Finally, for $X=\mathbb{S}_2$ the assertion follows from corollary~\ref{cor:krho_of_centS2}.
\end{proof}

\begin{lem} \label{lem:S_2_Euler}\label{lem:Ext_S_2}
For every $Y\in \CCC$ we have
 \begin{equation*}
  \dim \Ext^i_{\Gl(2)}(\widehat{{\mathbb{S}_2}},Y)=
  \begin{cases}m_{Y} & i=0\ ,\\
               0     & i>0
  \end{cases}\qquad\text{and}\qquad \chi(\mathbb{S}_2, Y)=m_Y\ ,
 \end{equation*}
 where $m_Y$ denotes the multiplicity of Whittaker models for $Y$.
\end{lem}
\begin{proof} We can assume that $Y$ is irreducible. Lemma~\ref{lem:Dual_Frob_Ext} and lemma~\ref{lem:cent_loc_S2} imply isomorphisms
\begin{equation*}
\Ext^i_{\CCC}(\widehat{\mathbb{S}_2},Y) \cong \Ext^i_{\CCC_1}(\mathbb{S}_1,({Y^\vee}|_{\Gl_a(1)})^\vee)\cong \Ext^i_{\CCC_0}(\C,j^!({Y^\vee}|_{\Gl_a(1)})^\vee)\ .
\end{equation*}
Since the category $\CCC_0$ of complex vector spaces is semisimple,
it only remains to be shown that $ \dim j^!({Y^\vee}|_{\Gl_a(1)})^\vee=m_Y$.
By Kirillov theory and since $\mathbb{S}_1$ is projective in $\CCC_{1}$, the $\Gl_a(1)$-module $(Y^\vee)|_{\Gl_a(1)}$ is an extension
$$1\to \mathbb{S}^{m_Y}\to (Y^\vee)|_{\Gl_a(1)}\to i_*i^*(Y^\vee)\to 1\ .$$
The contragredient of $\mathbb{S}$ is the $\Gl_a(1)$-module $C_b^\infty(k^\times)$ of degree one.
\end{proof}

\begin{lem}\label{lem:Euler_char_A}
For every $A\in\CCC^\fin_{\Gl(1)}$ and every $Y\in\CCC$, the Euler characteristic
$\chi(j_!i_*(A),Y)=0 $
vanishes.
\end{lem}
\begin{proof}
By lemma~\ref{lem:cent_loc_j_!i_*} and lemma~\ref{lem:Euler_char_additive}
we can assume that $A$ is a character and $Y$ is irreducible.
For $X=j_!i_*(A)$, lemma~\ref{lem:cent_loc_j_!i_*} implies
$$\widehat{X}=\zeta_{\mu}(X)=\ind_{B_{\Gl(2)}}^{\Gl(2)}(A\boxtimes \mu A^{-1})\ .$$
By passing to the contragredient, it suffices to show that $\Hom(Y^\vee,\widehat{X}^\vee)\cong\Ext^1(Y^\vee,\widehat{X}^\vee)$.
This holds by Frobenius reciprocity, see lemma~\ref{lem:Dual_Frob_Ext}.
\end{proof}

\begin{prop}\label{Finally}
For irreducible $\Pi\in\CCC_G(\omega)$ with a split Bessel model and not of type \nosf{VII}, \nosf{VIIIa}, \nosf{IXa} and for arbitrary $Y\in\CCC$,
the Euler characteristic is
\begin{equation*}
\chi(\overline{\Pi}, Y) = \chi(j_!j^!(\overline{\Pi}),  Y)=\chi(\mathbb{S}_2^{m_\Pi},Y)=m_Ym_\Pi
\end{equation*}
where $m_Y\in\{0,1\}$ is the dimension of Whittaker models for $Y$ and $m_\Pi\in\{0,1\}$ is the dimension of Whittaker models for $\Pi$.
\end{prop}

\begin{proof}
Recall that lemma~\ref{lem:Gelfand_Kazdhan} yields an exact sequence
$$0\to J\to \overline{\Pi}\to i_*(B) \to 0$$
with $B=i^*(\overline{\Pi})$ and $J=j_!j^!(\overline{\Pi})$.
By \cite[lemma~6.1]{W_Excep}, the associated long exact sequence of central specialization is
$$0 \to \zeta^\mu i_*(B)\to\zeta_\mu(J) \to \zeta_\mu{(\overline\Pi)} \to\zeta_\mu i_*(B)\to 0 \ .$$
Lemma~\ref{list} says $\zeta^\mu (i_*(B))\cong \zeta_\mu (i_*(B))$ and this implies $\chi(\overline{\Pi}, Y) = \chi(J,  Y)$  by lemma~\ref{lem:Euler_char_additive}. 
The functor $\zeta_\mu\circ j_!$ is exact by lemma~\ref{list}, so there is a short exact sequence where $A=i^*j^!(\overline{\Pi})$
$$0\to\zeta_\mu(\mathbb{S}_2^{m_\Pi})\to \zeta_\mu(J) \to \zeta_\mu\circ j_!i_* (A)\to 0\ .$$
Since $\chi(j_!i_*(A),Y)=0$ by lemma~\ref{lem:Euler_char_A},
we obtain $\chi(J,Y)={m_\Pi}\cdot\chi(\mathbb{S}_2,Y)$.
By lemma~\ref{lem:S_2_Euler} this equals $m_Y m_\Pi$.
\end{proof}

\begin{cor}\label{cor:central_specialization_generic}
For every generic irreducible $\Pi\in \CCC_{G}(\omega)$ and every generic irreducible $\pi\in \CCC_{\Gl(2)}(\mu)$,
we have $$1\leq\dim\Hom_{\Gl(2)}(\overline{\Pi},\pi)\leq2\ .$$
This dimension equals one if $\Pi$ is not of type \nosf{VII, VIIIa, IXa}.

Especially, there is a non-zero $\Gl(2)$-equivariant homomorphism $\overline{\Pi}\to\pi$ for every generic irreducible $\pi$,
so $\zeta_\mu(\overline{\Pi})$ does not have finite length.
\end{cor}

\begin{proof}
If $\Pi$ is not of type \nosf{VII, VIIIa, IXa}, the assertion holds by proposition~\ref{Finally}.
For $\Pi$ of type \nosf{VII, VIIIa, IXa}, lemma~\ref{lem:AB-sequence} yields an exact sequence $$0\to\mathbb{S}_2\to\overline{\Pi}\to i_*(B)\to 0\ ,$$
with cuspidal $B$. 
Note that $\Hom_{\Gl(2)}(\mathbb{S}_2,\pi)$ is one-dimensional by lemma~\ref{lem:S_2_Euler}, while Frobenius reciprocity implies that
$\Hom_{\Gl(2)}(B,\pi)$ is at most one-dimensional.
If it is non-zero the assertion is clear, so we assume there
is no non-zero $\Gl(2)$-equivariant morphism $B\to\pi$.
Central specialization satisfies $\widehat{B}=\zeta_\mu(i_*(B)) \cong \zeta^\mu(i_*(B))$ by lemma~\ref{list}.5 and because indecomposable extensions can only occur between isomorphic constituents of $B$.
This yields an exact sequence in $\CCC_{\Gl(2)}(\omega_\pi)$
$$\widehat{B}\stackrel{\delta}{\to}\zeta_\mu(\mathbb{S}_2)\to \zeta_\mu(\overline{\Pi})\to\widehat{B}\to0\quad\ .$$
Our assumption implies $\Hom(\widehat{B}/\ker(\delta),\pi)=0$\,, so the exact sequence
 $$0\to \widehat{B}/\ker(\delta)\to \zeta_\mu(\mathbb{S}_2) \to \coker(\delta)\to0$$
yields an isomorphism
\begin{equation*}
\C\cong\Hom(\zeta_\mu(\mathbb{S}_2),\pi)\cong \Hom(\coker(\delta),\pi)\ .
\end{equation*}
This is isomorphic to $\Hom(\zeta_\mu(\overline{\Pi}),\pi))$
because $\Ext^1_{\Gl(2)}(\zeta_\mu(B),\pi)=0$.
\end{proof}

This is in contrast to the situation of non-generic $\Pi$ where $\zeta_\mu(\overline\Pi)$ always has finite length, see~proposition~\ref{prop:central_specialization_non_generic}.
For the case of non-generic $\pi$ see theorem~\ref{MainL}.

\begin{table}
\caption{Central specializations for non-generic $\Pi$. \label{tab:central_specialization}}
\begin{footnotesize}
\begin{tabular}{lll}
\toprule
Type        & $\Pi\in\CCC_G$          &  $\zeta_\mu(\overline{\Pi})\in\CCC_{\Gl(2)}(\mu)$\\
\midrule
\nosf{IIb}  & $\chi_1\mathbf{1}\rtimes\sigma$ %
&  $\nu^{1/2}\chi_1\sigma\times\nu^{-1/2}\chi_1^{-1}\sigma^{-1}\mu$ for all $\mu$\\
\nosf{IIIb} & $\chi\rtimes\sigma\mathbf{1}$ for $\chi \neq\nu^{\pm1}$ %
&
$\begin{cases}
\nu^{1/2}\sigma M^c_{(\chi\times1:\chi\times1)} & \text{for}\ \mu=\nu\chi\sigma^2\\
\nu\sigma M_{(\St:1)}\oplus\nu^{1/2}\sigma(\chi\times\nu\chi^{-1}) & \text{for}\ \mu=\nu^2\sigma^2\text{ , }\chi^2\neq1\\
\nu\chi\sigma M_{(\St:1)}\oplus\nu^{1/2}\sigma(1\times\nu\chi^2) & \text{for}\ \mu=\nu^2\sigma^2\chi^2 \text{ , }\chi^2\neq1\\
\nu\chi\sigma M_{(\St:1)}\oplus\nu\sigma M_{(\St:1)} & \text{for } \mu=\nu^2\sigma^2\ \text{,}\ \chi^2=1\\
\nu^{1/2}\sigma\times\nu^{-1/2}\sigma^{-1}\mu 
& \text{else}\\
\qquad\oplus \nu^{1/2}\chi\sigma\times\nu^{-1/2}\chi^{-1}\sigma^{-1}\mu
\end{cases}$ \\
\nosf{IIIb} & $\nu\rtimes\sigma\mathbf{1}$ %
& $\begin{cases}\sigma(\nu^{1/2}\times\nu^{7/2})\oplus\nu^2\sigma M_{(\St:1)}&\text{for}\ \mu=\nu^4\sigma^2\\
\nu\sigma M^{c}_{(\St:1:\St:1)}&\text{for}\ \mu=\nu^2\sigma^2\\
\nu^{1/2}\sigma\times\nu^{-1/2}\sigma^{-1}\mu 
& \text{else}\\
\qquad\oplus \nu^{3/2}\sigma\times\nu^{-3/2}\sigma^{-1}\mu
\end{cases}$\\
\nosf{IVb}  & $L(\nu^2,\nu^{-1}\sigma\St)$ %
& $\begin{cases}\sigma M_{(1:\St)} & \text{for }\mu=\sigma^2\\\nu^{1/2}\sigma\times \nu^{-1/2 }\sigma^{-1}\mu&\text{else}\end{cases}$\\
\nosf{IVc}  & $L(\nu^{3/2}\St,\nu^{-3/2}\sigma)$ %
& $\begin{cases}
\nu^{-1/2}\sigma\times\nu^{9/2}\sigma\oplus \nu^2\sigma M_{(\St:1)} & \text{for}\ \mu=\nu^4\sigma^2\\
\nu^{-1/2}\sigma M^c_{(\nu^2\times1:\nu^2\times1)} & \text{for}\ \mu=\nu\sigma^2 \\
(\nu^{-1/2}\sigma\times\nu^{1/2}\sigma^{-1}\mu) 
& \text{else}\\
\qquad\oplus (\nu^{3/2}\sigma\times\nu^{-3/2}\sigma^{-1}\mu)
\end{cases}$ \\
\nosf{IVd}  & $\sigma\circ\lambda_G$         %
& $\begin{cases}(\sigma\circ\det)&\text{for}\ \mu=\sigma^2\\ 0 & \text{else}\end{cases}$\\
\nosf{Vb}   & $L(\nu^{1/2}\xi\St,\nu^{-1/2}\sigma)$ %
&  $\nu^{1/2}\sigma\times\nu^{-1/2}\sigma^{-1}\mu$ for all $\mu$\\
\nosf{Vd}   & $L(\nu\xi,\xi\rtimes\nu^{-1/2}\sigma)$%
& $\begin{cases}\nu\sigma(1\times\xi)& \text{for}\ \mu=\nu^2\xi\sigma^2\\0 & \text{else}\end{cases}$\\
\nosf{VIb}  & $\tau(T,\nu^{-1/2}\sigma)$        %
& $\begin{cases}\nu\sigma\St & \text{for}\ \mu=\nu^2\sigma^2\\0&\text{else}\end{cases}$            \\

\nosf{VIc}  & $L(\nu^{1/2}\St,\nu^{-1/2}\sigma)$
& $\begin{cases}\nu\sigma M_{(\St:1)}&\text{for}\ \mu=\nu^2\sigma^2\\\nu^{1/2}\sigma\times \nu^{-1/2}\sigma^{-1}\mu&\text{else}\end{cases}$\\
\nosf{VId}  & $L(\nu,1\rtimes\nu^{-1/2}\sigma)$ %
& $\begin{cases}%
\nu\sigma M_{(\St:1)}&\text{for}\ \mu=\nu^2\sigma^2\\ \nu^{1/2}\sigma\times\nu^{-1/2}\sigma^{-1}\mu&\text{else}\end{cases}$\\
\nosf{VIIIb}& $\tau(T,\pi_c)$                   %
& $\begin{cases}\nu\pi_c&\text{for}\ \mu=\nu^2\omega_{\pi_c}\\0&\text{else}     \end{cases}$\\
\nosf{IXb}  &$L(\nu\xi,\nu^{-1/2}\pi_c)$        %
&$\begin{cases}\nu^{1/2}\xi\pi_c &\text{for}\ \mu=\nu\omega_{\pi_c}\\0&\text{else}\end{cases}$\\ %
\nosf{XIb}  &$L(\nu^{1/2}\pi_c,\nu^{-1/2}\sigma)$%
& $\nu^{1/2}\sigma\times \nu^{-1/2}\sigma^{-1}\mu$ for all $\mu$                        \\
\bottomrule
\end{tabular}
\end{footnotesize}
$\ $

For the notation see proposition~\ref{prop:central_specialization_non_generic}.

\end{table}

\begin{table}\caption{Piatetski-Shapiro's spinor $L$-factor for split Bessel models \label{tab:regular_poles}}
\begin{footnotesize}
\begin{center}
 \begin{tabular}{llll}
\toprule
Type & $\Pi$                      &  $\rho$  & $L^{\mathrm{PS}}(s,\Pi,\Lambda,1)\ ,\ \  \Lambda=\rho\boxtimes\rho^\divideontimes$  split\\
\midrule
\nosf{I}    & $\chi_1\times\chi_2\rtimes\sigma$  & all & $L(s,\sigma)L(s,\chi_1\sigma)L(s,\chi_2\sigma)L(s,\chi_1\chi_2\sigma)$\\
\nosf{IIa}  & $\chi \St\rtimes\sigma$   & all          & $L(s,\sigma)L(s,\chi^2\sigma)L(s,\nu^{1/2}\chi\sigma)$\\
\nosf{IIb}  &$\chi\mathbf{1}\rtimes \sigma$& $\chi\sigma$     &$L(s,\sigma)L(s,\chi^2\sigma)L(s,\nu^{-1/2}\chi\sigma)L(s,\nu^{1/2}\chi\sigma)$\\
\nosf{IIIa} &$\chi\rtimes \sigma \St$    & all                & $L(s,\nu^{1/2}\chi\sigma)L(s,\nu^{1/2}\sigma)$\\
\nosf{IIIb} &$\chi\rtimes \sigma\mathbf{1}$ & $\sigma$, $\chi\sigma$        & $L(s,\nu^{-1/2}\chi\sigma)L(s,\nu^{-1/2}\sigma)L(s,\nu^{1/2}\chi\sigma)L(s,\nu^{1/2}\sigma)$\\     
\nosf{IVa}  &$\sigma \St_{G}$            & all                & $L(s,\nu^{3/2}\sigma)$\\
\nosf{IVb}  &$L(\nu^2,\nu^{-1}\sigma \St)$& $\sigma$          & $L(s,\nu^{3/2}\sigma)L(s,\nu^{-1/2}\sigma)$\\
\nosf{IVc}  &$L(\nu^{3/2} \St,\nu^{-3/2}\sigma)$& $\nu^{\pm1}\sigma$ & $L(s,\nu^{1/2}\sigma)L(s,\nu^{-3/2}\sigma)L(s,\nu^{3/2}\sigma)$\\
\nosf{IVd}  &$\sigma\mathbf{1}_G$        & none               & ---\\
\nosf{Va}   &$\delta([\xi,\nu\xi],\nu^{-1/2}\sigma)$&  all    & $L(s,\nu^{1/2}\sigma)L(s,\nu^{1/2}\xi\sigma)$\\
\nosf{Vb}   &$L(\nu^{1/2}\xi \St, \nu^{-1/2}\sigma)$ &$\sigma$& $L(s,\nu^{-1/2}\sigma)L(s,\nu^{1/2}\xi\sigma)L(s,\nu^{1/2}\sigma)$\\ 
\nosf{Vc}   &$L(\nu^{1/2}\xi \St, \nu^{-1/2}\xi\sigma)$ & $\xi\sigma$ & $L(s,\nu^{1/2}\sigma)L(s,\nu^{-1/2}\xi\sigma)L(s,\nu^{1/2}\xi\sigma)$\\
\nosf{Vd}   &$L(\nu\xi,\xi\rtimes\nu^{-1/2}\sigma)$ & none    & ---\\
\nosf{VIa}  &$\tau(S,\nu^{-1/2}\sigma)$  %
                                        & all                & $L(s,\nu^{1/2}\sigma)^2$\\
\nosf{VIb}  &$\tau(T,\nu^{-1/2}\sigma)$  & none               & ---\\
\nosf{VIc}  &$L(\nu^{1/2} \St,\nu^{-1/2}\sigma)$ & $\sigma$   & $L(s,\nu^{-1/2}\sigma)L(s,\nu^{1/2}\sigma)^2$\\
\nosf{VId}  &$L(\nu,1\rtimes\nu^{-1/2}\sigma)$ & $\sigma$     & $L(s,\nu^{-1/2}\sigma)^2L(s,\nu^{1/2}\sigma)^2$\\
\nosf{VII}  &$\chi\rtimes\pi$            & all                & $1$\\
\nosf{VIIIa}&$\tau(S,\pi)$               & all                & $1$\\
\nosf{VIIIb}&$\tau(T,\pi)$               & none               & ---\\
\nosf{IXa}  &$\delta(\nu\xi,\nu^{-1/2}\pi)$& all              & $1$\\
\nosf{IXb}  & $L(\nu\xi,\nu^{-1/2}\pi)$  & none               & ---\\
\nosf{X}    &$\pi\rtimes\sigma$           %
                                        & all          & $L(s,\sigma)L(s,\omega_{\pi}\sigma)$\\
\nosf{XIa}  &$\delta(\nu^{1/2}\pi,\nu^{-1/2}\sigma)$&%
                                        all    & $L(s,\nu^{1/2}\sigma)$\\
\nosf{XIb}  &$L(\nu^{1/2}\pi, \nu^{-1/2}\sigma)$ & $\sigma$   & $L(s,\nu^{-1/2}\sigma)L(s,\nu^{1/2}\sigma)$\\
            & cuspidal generic           & all                & $1$\\
            & cuspidal non-generic       & none               & ---\\
\bottomrule
\end{tabular}
\end{center}
\end{footnotesize}
For every irreducible smooth representation $\Pi$ of $\GSp(4,k)$ with central character $\omega$, the column $\rho$ lists the smooth characters such that the character $\Lambda=\rho\boxtimes\rho^\divideontimes$ of $\widetilde{T}\cong k^\times\times k^\times$ where $\rho^\divideontimes=\omega\rho^{-1}$
yields a split Bessel model for $\Pi$.
The right column lists Piatetski-Shapiro's spinor $L$-factors attached to this split Bessel model.
For non-cuspidal $\Pi$ we use the classification symbols of \cite{Sally_Tadic} and \cite{Roberts-Schmidt}.
\end{table}

\begin{footnotesize}
\bibliographystyle{amsalpha}

\end{footnotesize}

\vskip 18 pt
\begin{footnotesize}
\centering{Mirko R\"osner\\ Mathematisches Institut, Universit\"at Heidelberg\\ Im Neuenheimer Feld 205, 69120 Heidelberg\\ email: mroesner@mathi.uni-heidelberg.de}

\vskip 8 pt
\centering{Rainer Weissauer\\ Mathematisches Institut, Universit\"at Heidelberg\\ Im Neuenheimer Feld 205, 69120 Heidelberg\\ email: weissauer@mathi.uni-heidelberg.de}

\end{footnotesize}

\end{document}